\documentclass[nointlimits,11pt,oneside]{amsart}
\usepackage{amssymb,cases,enumerate}
\usepackage{color}
\usepackage{hyperref,mathscinet}
\usepackage{natbib}
\usepackage{mathscinet}

\usepackage[%
	a4paper,
	total={16cm,23cm},
	top=3cm,
	marginparsep=2pt]
{geometry}

\newcommand{\N}{\mathbb{N}}
\newcommand{\R}{\mathbb{R}}
\newcommand{\rn}{\R^n}
\newcommand{\M}{\mathcal M}
\newcommand{\Mpl}{\M_+}
\renewcommand{\d}{{\fam0 d}}
\newcommand{\vr}{\varrho}

\DeclareMathOperator*{\esssup}{ess\,sup}
\DeclareMathOperator{\Tr}{T_{\nu}}

\theoremstyle{plain}
\newtheorem{theorem}{Theorem}[section]

\newtheorem{proposition}[theorem]{Proposition}
\theoremstyle{definition}
\newtheorem{remark}[theorem]{Remark}

\newtheorem{question}{Question}
\numberwithin{equation}{section}

\usepackage[textsize=tiny,textwidth=2cm,colorinlistoftodos]{todonotes}

\makeatletter
\def\paragraph{\bigskip\@startsection{paragraph}{4}%
  \z@\z@{-\fontdimen2\font}%
  {\normalfont\bfseries}}
\makeatother

\begin{document}

\title{Compactness of Sobolev embeddings and decay of norms}
\author{Jan Lang, Zden\v ek Mihula and Lubo\v s Pick}

\address{Jan Lang, Department of Mathematics,
The Ohio State University,
231 West 18th Avenue,
Columbus, OH 43210-1174}
\email{lang.162@osu.edu}
\urladdr{0000-0003-1582-7273}

\address{Zden\v{e}k Mihula, Department of Mathematical Analysis, Faculty of Mathematics and
Physics, Charles University, Sokolovsk\'a~83,
186~75 Praha~8, Czech Republic
--AND-- Department of Mathematics, Faculty of Electric Engineering, Czech Technical University in Prague, Technick\'a~2,
166 27 Prague~6, Czech Republic}
\email{mihulaz@karlin.mff.cuni.cz}
\email{mihulzde@fel.cvut.cz}
\urladdr{0000-0001-6962-7635}

\address{Lubo\v s Pick, Department of Mathematical Analysis, Faculty of Mathematics and
Physics, Charles University, Sokolovsk\'a~83,
186~75 Praha~8, Czech Republic}
\email{pick@karlin.mff.cuni.cz}
\urladdr{0000-0002-3584-1454}


\thanks{This research was partly supported by the grant P201-18-00580S of the Czech Science Foundation; by the Charles University, project GA UK No.~1056119; by Charles University Research program No.~UNCE/SCI/023; by the project OPVVV CAAS CZ.02.1.01/0.0/0.0/16\_019/0000778.}

\begin{abstract}
We investigate the relationship between the compactness of embeddings of Sobolev spaces built upon rearrangement-invariant spaces into rearrangement-invariant spaces endowed with $d$-Ahlfors measures under certain restriction on the speed of its decay on balls. We show that the gateway to compactness of such embeddings, while formally describable by means of optimal embeddings and almost-compact embeddings, is quite elusive. It is known that such a Sobolev embedding is not compact when its target space has the optimal fundamental function. We show that, quite surprisingly, such a target space can actually be ``fundamentally enlarged'', and yet the resulting embedding remains noncompact. In order to do that, we develop two different approaches. One is based on enlarging the optimal target space itself, and the other is based on enlarging the Marcinkiewicz space corresponding to the optimal fundamental function.
\end{abstract}

\date{\today}

\maketitle

\setcitestyle{numbers}
\bibliographystyle{plainnat}

\section{Introduction}

Compact embeddings of function spaces containing weakly differentiable functions defined on subdomains of an Euclidean space into other function spaces constitute an important technique that is widely applicable when solutions to partial differential equations are sought by functional-analytic or variational methods. Such embeddings are particularly handy for showing the discreteness of the spectra of linear elliptic partial differential operators defined on bounded domains.

The most classical result on the compactness of a Sobolev embedding is the Rellich--Kondrachov theorem, which originated in a lemma of Rellich~\cite{Re:30} and was proved specifically for Sobolev spaces by
Kondrachov~\cite{Ko:45}. It is often used in the form stating that, given $n\in\N$, $n\geq2$ (we shall assume this implicitly throughout the paper), $p\in[1,n]$ and a bounded Lipschitz domain $\Omega\subseteq\rn$, the Sobolev space $W^{1,p}(\Omega)$ is compactly embedded into the Lebesgue space $L^{q}(\Omega)$ for any $q\in[1,\frac{np}{n-p})$ (the fraction $\frac{np}{n-p}$ is to be interpreted as $\infty$ when $p=n$). Possibly the most natural way of proving the Rellich--Kondrachov theorem is based on the fact that a bounded set in $W^{1,p}(\Omega)$ is equiintegrable in $L^{q}(\Omega)$, that is, given $\varepsilon>0$, there always exists a $\delta>0$ such that for every subset $E$ of $\Omega$ of measure not exceeding $\delta$ one has
\begin{equation*}
    \sup_{\|u\|_{W^{1,p}(\Omega)}\le1} \|u\chi_{E}\|_{L^{q}(\Omega)} < \varepsilon,
\end{equation*}
where $\chi_E$ stands for the characteristic function of $E$. There are several ways to achieve this fact. One of the most successful ones is based on a combination of the (in some sense) optimal Sobolev embedding and a so-called almost-compact embedding between function spaces on the target side of the embedding relation. We shall now describe this technique in more detail.

Roughly speaking, a space $Y(\Omega)$ is entitled to be called the \textit{optimal target space} for a given space $X(\Omega)$ in the Sobolev embedding $W^1X(\Omega)\hookrightarrow Y(\Omega)$ if it cannot be replaced by any essentially smaller space from a specified category of function spaces. A precise specification of the pool of competing spaces is important. For example, if $p\in[1,n)$ and $\Omega$ is a bounded Lipschitz domain in $\R^{n}$, then, in the classical Sobolev embedding
\begin{equation}\label{E:classical-cobolev}
    W^{1,p}(\Omega)\hookrightarrow L^{\frac{np}{n-p}}(\Omega),
\end{equation}
the space $L^{\frac{np}{n-p}}(\Omega)$ is the optimal range partner for $L^{p}(\Omega)$ \textit{in the category of Lebesgue spaces} because it cannot be replaced by any essentially smaller \textit{Lebesgue} space. While the embedding
\begin{equation*}
    L^{\frac{np}{n-p}}(\Omega) \hookrightarrow L^{q}(\Omega)
\end{equation*}
is continuous, it is not compact. The next step is to observe that the embedding, while noncompact, is \textit{almost compact} in the sense that
\begin{equation*}
    \lim_{k\to\infty}\sup_{\|u\|_{L^{\frac{np}{n-p}}(\Omega)}\leq1} \|u\chi_{E_k}\|_{L^{q}(\Omega)} =0
\end{equation*}
for every sequence $\{E_k\}$ of measurable subsets of $\Omega$ satisfying $E_k\searrow\emptyset$ a.e.~in $\Omega$. These observations can be summarized in a chain of embeddings, namely
\begin{equation}\label{E:scheme}
    W^{1,p}(\Omega) \hookrightarrow L^{\frac{np}{n-p}}(\Omega) \stackrel{*}{\hookrightarrow} L^{q}(\Omega),
\end{equation}
where the symbol $\stackrel{*}{\hookrightarrow}$ denotes an almost-compact embedding. Not surprisingly, \eqref{E:scheme} implies that every bounded set in $W^{1,p}(\Omega)$ is equiintegrable in $L^{q}(\Omega)$. The fact that the combination of the two relations in~\eqref{E:scheme} guarantees that
\begin{equation*}
    W^{1,p}(\Omega) \hookrightarrow \hookrightarrow L^{q}(\Omega),
\end{equation*}
where the symbol $\hookrightarrow\hookrightarrow$ denotes a compact embedding, is known. For instance, it is explicitly proved, in a more general setting, in~\cite[Theorem~3.2]{Sla:12}. Interestingly, one obtains the almost-compact embedding
\begin{equation*}
    L^{\frac{np}{n-p}}(\Omega) \stackrel{*}{\hookrightarrow} L^{q}(\Omega)
\end{equation*}
almost for free: suppose that $q\in[1,\frac{np}{n-p})$ and $E_k\searrow \emptyset$ a.e., then we get by H\"older's inequality that, for every function $u$ in the closed unit ball of $L^{\frac{np}{n-p}}(\Omega)$,
\begin{equation}\label{E:fundamental-lebesgue}
    \|u\chi_{E_k}\|_{L^{q}(\Omega)} \le \|u\|_{L^{\frac{np}{n-p}}(\Omega)}\|\chi_{E_k}\|_{L^{\frac{qnp}{np-nq+pq}}(\Omega)} \le |E_k|^{\frac1{q}-\frac1{p}+\frac1{n}}\to 0\quad\text{as $k\to\infty$,}
\end{equation}
where $|E_k|$ denotes the $n$-dimensional Lebesgue measure of $E_k$. To explore the scheme illustrated by~\eqref{E:scheme} any deeper, we need, however, finer classes of function spaces than that of Lebesgue spaces.

Although the classical theory works almost solely with Lebesgue spaces, there are other, more complicated, function spaces that are also of considerable interest. Important generalizations of Lebesgue spaces are
Lorentz spaces and Orlicz spaces. While Lorentz spaces constitute a useful tool for certain fine tuning of Lebesgue spaces, Orlicz spaces have been successfully used when either more rapid or slower than power growth of functions is needed (cf.~\cite{Go:74}, \cite{GiTr:01}, \cite{Ci:96}). Allowing these types of spaces has a considerable impact on the quality of Sobolev embeddings. For example, in~\eqref{E:classical-cobolev}, the target space is optimal as a Lebesgue space, but it is not optimal as a Lorentz space, because it can be replaced by a strictly smaller Lorentz space $L^{\frac{np}{n-p},p}(\Omega)$. The situation is even more interesting when the Sobolev space $W^{1,n}(\Omega)$ is in play, in which the
degree of integrability coincides with the dimension of the underlying Euclidean space. In that case, there does not exist any optimal Lebesgue target space, but there does exist an optimal Orlicz space, namely the celebrated
Zygmund class $\exp L^{n'}(\Omega)$ ($n'=\frac{n}{n-1}$), a result that is nowadays considered classical and that goes back to Trudinger~\cite{Tr:67}, Poho\v{z}aev~\cite{Po:65} and Judovi\v{c}~\cite{Yu:61}.
However, neither Lorentz nor Orlicz spaces hold the key to all answers, because even this space can be improved. It can be replaced by the Lorentz--Zygmund space $L^{\infty,n;-1}(\Omega)$, which is strictly smaller than $\exp L^{n'}(\Omega)$ and which has been surfacing in various contexts and also in various disguises, see, for example, \cite{CP,Ma:11,BW,Ha}. The scale of Lorentz--Zygmund spaces, in some sense a meeting point of Orlicz and Lorentz families of spaces, was introduced in~\cite{BeRu:80} and has been later generalized in numerous ways, e.g.~\cite{OP, EvGoOp:18, EdKePi:00}.


In order to avoid technical difficulties that each individual scale of spaces inevitably brings, it is advisable to make use of a common framework that encompasses all of these function spaces. Experience shows that probably the most suitable one is that of the \textit{rearrangement-invariant spaces} (for precise definitions see \hyperref[sec:prel]{Section~\ref*{sec:prel}}). This category of function spaces is naturally built on the procedure of symmetrization, but it also involves spaces whose original definitions do not rely on rearrangement techniques, such as Lebesgue or Orlicz spaces.

In the framework of rearrangement-invariant spaces, Sobolev embeddings have been studied heavily over the past two decades. The key technique is the so-called \textit{reduction principle}, which enables us to reformulate equivalently a difficult problem involving differential operators and functions of many variables in the form of a question concerning boundedness of an integral operator involving functions acting on an interval. These advances, introduced in~\cite{EdKePi:00} and then further developed in many works, see e.g.~\cite{KePi:06,CiPiSl:15}, paved the way for studying deeper properties of Sobolev embeddings, a pivotal example of which is compactness, see~\cite{KerPi:08,Sla:15,Sla:12,CaMi:19}.

The results of Slav\'{\i}kov\'a~\cite{Sla:12,Sla:15} showed that the two-step method described above in connection with Lebesgue spaces is extendable to the general setting of rearrangement-invariant spaces. There are, however, some pitfalls. In particular, if one can prove an analogue of~\eqref{E:scheme} in the form
\begin{equation*}
    W^{1}X(\Omega) \hookrightarrow Y_X(\Omega) \stackrel{*}{\hookrightarrow} Z(\Omega),
\end{equation*}
in which $X(\Omega)$ and $Z(\Omega)$ are rearrangement-invariant spaces and $Y_X(\Omega)$ is the optimal (the smallest) rearrangement-invariant space $Y(\Omega)$ rendering $W^{1}X(\Omega) \hookrightarrow Y(\Omega)$ true (such a rearrangement-invariant space always exists; we shall comment on that later), then the desired compact embedding
\begin{equation}\label{E:compact-embedding-general}
    W^{1}X(\Omega) \hookrightarrow \hookrightarrow Z(\Omega)
\end{equation}
follows. What is not, however, clear at all is whether one can use an analogue of~\eqref{E:fundamental-lebesgue} to get \eqref{E:compact-embedding-general}. Such an analogue would inevitably call to play the so-called \textit{fundamental function} of a rearrangement-invariant space $X(\Omega)$. Its fundamental function $\varphi_X$ is defined on $[0,|\Omega|)|$ as
\begin{equation*}
    \varphi_X(t)=\|\chi_E\|_{X},\ \text{$t\in[0,|\Omega|)$, where $E\subseteq\Omega$ is any subset satisfying $|E|=t$.}
\end{equation*}
A proper analogue of~\eqref{E:fundamental-lebesgue} would be something like
\begin{equation}\label{E:fundamental-general}
    \lim_{t\to0_+} \frac{\varphi_{Y_X}(t)}{\varphi_Z(t)} =  0,
\end{equation}
but the question is, does \eqref{E:fundamental-general} imply \eqref{E:compact-embedding-general}? We shall see that the situation is much more complicated when we do not restrict ourselves to the class of Lebesgue spaces.

The concept of the fundamental function leads us to the \textit{Marcinkiewicz spaces} $M_{\varphi}(\Omega)$
and the \textit{Lorentz endpoint spaces} $\Lambda_{\varphi}(\Omega)$ (see \hyperref[sec:prel]{Section~\ref*{sec:prel}} for their definitions).
It is known that $M_{\varphi}(\Omega)$ and $\Lambda_{\varphi}(\Omega)$ are the biggest and the smallest, respectively, rearrangement-invariant spaces with the fixed fundamental function $\varphi$.
The embedding
\begin{equation}\label{E:optimal-r.i.-embedding}
    W^{1}X(\Omega)\hookrightarrow Y_X(\Omega)
\end{equation}
is known to never be compact, regardless of the choice of $X(\Omega)$, but what if we replace $Y_X(\Omega)$ with the Marcinkiewicz space $M_{Y_X}(\Omega)$, the largest space having the same fundamental function as $Y_X(\Omega)$? Since $Y_X(\Omega)$ is embedded into $M_{Y_X}(\Omega)$, we plainly have that
\begin{equation}\label{E:optimal-marc-embedding}
    W^{1}X(\Omega)\hookrightarrow M_{Y_X}(\Omega).
\end{equation}
It turns out, however, that, even though \eqref{E:optimal-marc-embedding} has a possibly larger target space (hence the embedding is possibly weaker) than~\eqref{E:optimal-r.i.-embedding}, it is still never compact (we shall comment on that in more detail later). At this stage, we are left with the question, how much larger than $M_{Y_X}(\Omega)$ can a target space $Y(\Omega)$ be in order to guarantee that the embedding $W^1X(\Omega)\hookrightarrow Y(\Omega)$ is still not compact?

All of the observations that we made lead us to three natural questions, which will be formulated soon. However, in order to ensure that our results are applicable to a large number of different Sobolev-type embeddings, which are often considered and studied separately, we actually consider much more general embeddings of Sobolev-type spaces (even of higher orders) into function spaces built upon quite general measure spaces. We say that a finite Borel measure $\nu$ on $\overline{\Omega}$ is a \emph{$d$-Ahlfors measure}, where $d\in(0,n]$, if
\begin{equation*}
\sup\limits_{x\in\rn, r>0}\frac{\nu\left(B_r(x)\cap\overline{\Omega}\right)}{r^d}<\infty,
\end{equation*}
where $B_r(x)$ is the open ball in $\rn$ centered at $x$ with radius $r$, and there is a point $x_0\in\overline{\Omega}$ and $R>0$ such that
\begin{equation*}
\inf\limits_{r\in(0,R]}\frac{\nu\left(B_r(x_0)\cap\overline{\Omega}\right)}{r^d}>0.
\end{equation*}
We shall consider Sobolev-type embeddings having the form $W^mX(\Omega)\hookrightarrow Y(\overline{\Omega}, \nu)$, where $W^mX(\Omega)$ is the $m$-th order, $m\in\N$, Sobolev-type space built upon a rearrangement-invariant space $X(\Omega)$, and $Y(\overline{\Omega}, \nu)$ is a rearrangement-invariant space on $\overline{\Omega}$ endowed with a $d$-Ahlfors measure $\nu$ (for more detail, see
\hyperref[sec:prel]{Section~\ref*{sec:prel}}). Such Sobolev-type embeddings and their compactness were recently studied in \citep{CPS:20, CPS_OrlLor:20, CM:20}. This general setting encompasses, for example, not only the standard Sobolev
embeddings ($\nu=\lambda_n$, the $n$-dimensional Lebesgue measure on $\Omega$) but also Sobolev trace embeddings onto $d$-dimensional submanifolds ($\nu=\mathcal H^d\rvert_{\Omega_d}$, where $\Omega_d$ is a $d$-dimensional
Riemannian submanifold) and boundary trace embeddings ($\nu=\mathcal H^{n-1}\rvert_{\partial\Omega}$) as well as some weighted Sobolev embeddings (e.g.~$\d\nu(x)=|x-z|^{d-n}\,\d x$, where $z\in\overline{\Omega}$ is a fixed
point). Now we can formulate the principal questions that we tackle in this paper.

Let $m<n$ and $d\in[n-m,n]$. Let $X(\Omega)$ be a rearrangement-invariant space such that $X(\Omega)\not\subseteq L^{\frac{n}{m},1}(\Omega)$. Let $Y_X(\overline{\Omega}, \nu)$ be the optimal rearrangement-invariant target space in $W^mX(\Omega)\hookrightarrow Y(\overline{\Omega}, \nu)$ (its existence is guaranteed by \citep[Theorem~4.4]{CPS:20}). It was proved in \citep[Theorem~4.1]{CM:20} (cf.~\citep[Theorem~4.1]{Sla:15}, \citep[Theorem~5.1]{KerPi:08}) that, if $Y(\overline{\Omega}, \nu)\neq L^\infty(\overline{\Omega}, \nu)$, then $W^mX(\Omega)\hookrightarrow Y(\overline{\Omega}, \nu)$ is compact if and only if $Y_X(\overline{\Omega}, \nu)\overset{*}{\hookrightarrow} Y(\overline{\Omega}, \nu)$. It follows from this characterization combined with \eqref{prel:almostcompembnecessary} that the optimal embedding $W^mX(\Omega)\hookrightarrow Y_X(\overline{\Omega}, \nu)$ is never compact. Actually, since $Y_X(\overline{\Omega}, \nu)$ and $M_{Y_X}(\overline{\Omega}, \nu)$ have the same fundamental function, even the embedding $W^mX(\Omega)\hookrightarrow M_{Y_X}(\overline{\Omega}, \nu)$ is never compact. While $M_{Y_X}(\overline{\Omega}, \nu)$ is the biggest rearrangement-invariant space on the same fundamental scale as $Y_X(\overline{\Omega}, \nu)$, it is not the biggest rearrangement-invariant target space that renders the embedding noncompact. Surprising as it may appear, there is, in general, no such a space. We shall show that the set of the target spaces that renders a Sobolev embedding noncompact has, roughly speaking, no biggest element. The remarks made in this paragraph lead us to two possible ideas of how to construct bigger, noncompact target spaces.

The first idea is to enlarge the Marcinkiewicz space $M_{Y_X}(\overline{\Omega}, \nu)$ to a space $Y(\overline{\Omega}, \nu)$ in such a way that the embedding $W^mX(\Omega)\hookrightarrow Y(\overline{\Omega}, \nu)$ is still not compact.
\begin{question}\label{q:noncompactMarc}
Is there a rearrangement-invariant space $Y(\overline{\Omega}, \nu)$ with the following properties:

\begin{itemize}
  \item $W^mX(\Omega)\hookrightarrow Y(\overline{\Omega}, \nu)$ non-compactly,
  \item $M_{Y_X}(\overline{\Omega}, \nu)\subsetneq Y(\overline{\Omega}, \nu)$, where $M_{Y_X}(\overline{\Omega}, \nu)$ is the Marcinkiewicz space corresponding to $Y_X(\overline{\Omega}, \nu)$?
\end{itemize}
\end{question}

We shall give a comprehensive answer to \hyperref[q:noncompactMarc]{Question~\ref*{q:noncompactMarc}}, in which we make substantial use of a construction that ensures that, given a Marcinkiewicz space $M_\varphi$, we can construct a Marcinkiewicz space $M_\psi$ such that $M_\varphi\subsetneq M_\psi$ having the crucial property that $M_\varphi$ is not almost-compactly embedded into $M_\psi$. The construction does not, however, guarantee that $M_\psi$ is ``fundamentally bigger'' than $M_\varphi$, that is, $\lim\limits_{t\to0^+}\frac{\psi(t)}{\varphi(t)}=0$. In fact, the limit does not exist at all. It is thus natural to search for a ``fundamentally bigger'' space in the next step. This leads us to a new problem, closely related to \hyperref[q:noncompactMarc]{Question~\ref*{q:noncompactMarc}} but more difficult.

\begin{question}\label{q:noncompactfundamentallybigger1}
Is there a rearrangement-invariant space $Y(\overline{\Omega}, \nu)$ with the following properties:

\begin{itemize}
  \item $W^mX(\Omega)\hookrightarrow Y(\overline{\Omega}, \nu)$ non-compactly,
  \item $M_{Y_X}(\overline{\Omega}, \nu)\subsetneq Y(\overline{\Omega}, \nu)$,
	\item $\lim\limits_{t\to0^+}\frac{\varphi_Y(t)}{\varphi_{Y_X}(t)}=0$?
\end{itemize}
\end{question}
The difference between \hyperref[q:noncompactMarc]{Question~\ref*{q:noncompactMarc}} and \hyperref[q:noncompactfundamentallybigger1]{Question~\ref*{q:noncompactfundamentallybigger1}} is that, in the latter, the space $Y(\overline{\Omega}, \nu)$ is required to be ``fundamentally bigger'' than $Y_X(\overline{\Omega}, \nu)$. We deal with the first two questions in \hyperref[sec:q1q2]{Section~\ref*{sec:q1q2}}.

The other approach is to enlarge the optimal space $Y_X(\overline{\Omega}, \nu)$ to a space $Y(\overline{\Omega}, \nu)$ in such a way that the embedding $W^mX(\Omega)\hookrightarrow Y(\overline{\Omega}, \nu)$ is still not compact. This leads to the following question, which we deal with in \hyperref[prel:q3]{Section~\ref*{prel:q3}}.

\begin{question}\label{q:noncompactfundamentallybigger2}
Is there a rearrangement-invariant space $Y(\overline{\Omega}, \nu)$ with the following properties:

\begin{itemize}
  \item $W^mX(\Omega)\hookrightarrow Y(\overline{\Omega}, \nu)$ non-compactly,
  \item $Y_X(\overline{\Omega}, \nu)\subsetneq Y(\overline{\Omega}, \nu)$,
	\item $\lim\limits_{t\to0^+}\frac{\varphi_Y(t)}{\varphi_{Y_X}(t)}=0$?
\end{itemize}
\end{question}

On the one hand, since we no longer require that $Y(\overline{\Omega}, \nu)$ contains the Marcinkiewicz space $M_{Y_X}(\overline{\Omega}, \nu)$, such a space $Y(\overline{\Omega}, \nu)$ can be closer to the optimal space $Y_X(\overline{\Omega}, \nu)$; thus it may appear that this approach leads to weaker results. On the other hand, by dropping the requirement that $Y(\overline{\Omega}, \nu)$ contains the Marcinkiewicz space $M_{Y_X}(\overline{\Omega}, \nu)$, we lose a great deal of information; thus the second approach brings in a lot of technical complications, which we do not face when following the first idea. Furthermore, the results in \hyperref[prel:q3]{Section~\ref*{prel:q3}} are usable in situations not covered by the results in \hyperref[sec:q1q2]{Section~\ref*{sec:q1q2}}. Therefore, \hyperref[sec:q1q2]{Section~\ref*{sec:q1q2}} and \hyperref[prel:q3]{Section~\ref*{prel:q3}} complement each other rather than one extending the other. While pursuing the second approach, we also obtain a useful result of independent interest, namely \hyperref[prop:optimalpair]{Proposition~\ref*{prop:optimalpair}}, which characterizes when the spaces in a Sobolev embedding $W^mX(\Omega)\hookrightarrow Y(\overline{\Omega}, \nu)$ are mutually optimal.

Before we start searching for answers, we provide some comments on the restrictions on the parameters $d,m,n$ and on the space $X(\Omega)$. First, when $m\geq n$, the rearrangement-invariant setting is not well suited to capturing fine details of corresponding Sobolev embeddings, because the space $W^{m}X(\Omega)$ is embedded into $L^{\infty}(\overline{\Omega}, \nu)$, the smallest rearrangement-invariant space over $(\overline{\Omega}, \nu)$, no matter what $X(\Omega)$ is and $L^{\infty}(\overline{\Omega}, \nu)$ is almost-compactly embedded into any rearrangement-invariant space over $(\overline{\Omega}, \nu)$ that is different from $L^{\infty}(\overline{\Omega}, \nu)$ (\citep[Theorem~5.2]{Sla:12}). In this case, a more suitable class of potential target spaces consists of various function spaces measuring smoothness and/or oscillation (rather than size). Such research, while of great interest, goes beyond the scope of this paper.

Next, the assumption $X(\Omega)\not\subseteq L^{\frac{n}{m},1}(\Omega)$ is actually completely natural. Indeed, if $X(\Omega)\subseteq L^{\frac{n}{m},1}(\Omega)$, then $Y_X(\overline{\Omega}, \nu)=L^\infty(\overline{\Omega}, \nu)$ (\citep[Theorem~3.1]{CPS_OrlLor:20}). Hence $W^mX(\Omega)\hookrightarrow Y(\overline{\Omega}, \nu)$ is always compact whenever $Y(\overline{\Omega}, \nu)\neq L^\infty(\overline{\Omega}, \nu)$ thanks to \citep[Theorem~4.1]{CM:20} combined with \citep[Theorem~5.2]{Sla:12}. Therefore, the questions are of no interest if $X(\Omega)\subseteq L^{\frac{n}{m},1}(\Omega)$.

Finally, the assumption $d\ge n-m$ is technical in nature. It goes back to~\cite{CPS:20,CPS_OrlLor:20}, where it was discovered that a balance condition between $d,m$ and $n$ constitutes a threshold that divides Sobolev embeddings with respect to $d$-Ahlfors measures into two completely different groups. It turns out that the so-called fast-decaying measures (that is, those having $d\ge n-m$) behave naturally, while their counterparts, the slowly-decaying measures ($d<n-m$), bring some rather unexpected technical anomalies, which we prefer to avoid here.


We collect the background material used in this paper and fix the notation in \hyperref[sec:prel]{Section~\ref*{sec:prel}}, which is quite lengthy because we prefer this paper to be as self-contained as possible. Readers familiar with rearrangement-invariant spaces might want to skim over the section.

\section{Background material}\label{sec:prel}
Throughout the paper, the relations $\lq\lq \lesssim "$ and $\lq\lq \gtrsim"$ between two positive expressions mean that the former is bounded by the latter and vice versa, respectively, up to a multiplicative constant independent of all important quantities in question. When both these relations hold at the same time (with possibly different constants), we write $\lq\lq \approx "$.

Throughout this section, let $(R, \mu)$ be a $\sigma$-finite, nonatomic measure space.
We set
\begin{align*}
\M(R, \mu)&= \{f: \text{$f$ is a $\mu$-measurable function on $R$ with values in $[-\infty,\infty]$}\}\\
\intertext{and}
\Mpl(R, \mu)&= \{f \in \M(R, \mu)\colon f \geq 0\}.
\end{align*}
The \emph{nonincreasing rearrangement} $f^* \colon  [0,\infty) \to [0, \infty ]$ of a function $f\in \M(R, \mu)$  is
defined as
\begin{equation*}
f^*(t)=\inf\{\lambda\in(0,\infty)\colon\mu\left(\{x\in R\colon|f(x)|>\lambda\}\right)\leq t\},\ t\in[0,\infty).
\end{equation*}
The \emph{maximal nonincreasing rearrangement} $f^{**} \colon (0,\infty) \to [0, \infty ]$ of a function $f\in \M(R, \mu)$  is
defined as
\begin{equation*}
f^{**}(t)=\frac1t\int_0^ t f^{*}(s)\,\d s,\ t\in(0,\infty).
\end{equation*}
If there is any possibility of misinterpretation, we use the more explicit notations $f^*_\mu$ and $f^{**}_\mu$ instead of $f^*$ and $f^{**}$, respectively, to stress what measure the rearrangements are taken with respect to.

If $|f|\leq |g|$ $\mu$-a.e.\ in $R$, then $f^*\leq g^*$. The operation $f\mapsto f^*$ is neither subadditive nor multiplicative. The lack of subadditivity of the operation
of taking the nonincreasing rearrangement is, up to some extent, compensated by the following fact (\cite[Chapter~2,~(3.10)]{BS}): for every
$t\in(0,\infty)$ and every $f,g\in\mathcal M(R, \mu)$, we have that
\begin{equation*}
\int_0^ t(f +g)^{*}(s)\,\d s\leq\int_0^ tf^{*}(s)\,\d s + \int_0^ tg^{*}(s)\,\d s.
\end{equation*}
This inequality can be also written in the form
\begin{equation*}
(f+g)^{**}\leq f^{**}+g^{**}.
\end{equation*}
A fundamental result in the theory of Banach function spaces is the \emph{Hardy lemma} (\citep[Chapter~2, Proposition~3.6]{BS}), which states that, if two nonnegative measurable functions $f,g$ on $(0,\infty)$ satisfy
\begin{align}
\int_0^{t}f(s)\,\d s&\leq \int_0^ tg(s)\,\d s\nonumber
\intertext{for all $t\in(0,\infty)$, then, for every nonnegative, nonincreasing function $h$ on $(0,\infty)$, one has}
\int_0^{\infty}f(s)h(s)\,\d s&\leq \int_0^ {\infty}g(s)h(s)\,\d s.\label{prel:hardy-lemma}
\end{align}

An important fact concerning rearrangements is the \emph{Hardy-Littlewood inequality} (\citep[Chapter~2, Theorem~2.2]{BS}), which asserts that, if $f, g \in\M(R, \mu)$,
then
\begin{equation}\label{prel:HL}
\int _R |fg| \,\d\mu \leq \int _0^{\infty} f^*(t) g^*(t)\,\d t.
\end{equation}

If $(R, \mu)$ and $(S, \nu)$ are two (possibly different) $\sigma$-finite measure spaces, we say that functions $f\in \M(R, \mu)$ and $g\in\M(S,\nu)$ are \emph{equimeasurable}, and we write $f\sim g$, if $f^*=g^*$ on $(0,\infty)$. Note that $f$ and $f^*$ are equimeasurable.

A functional $\vr\colon  \Mpl (R, \mu) \to [0,\infty]$ is called a \emph{Banach function norm} if, for all $f$, $g$ and $\{f_j\}_{j\in\N}$ in $\M_+(R, \mu)$, and every $\lambda \geq0$, the following properties hold:
\begin{enumerate}[(P1)]
\item $\vr(f)=0$ if and only if $f=0$;
$\vr(\lambda f)= \lambda\vr(f)$; $\vr(f+g)\leq \vr(f)+ \vr(g)$;
\item $  f \le g$ a.e.\  implies $\vr(f)\le\vr(g)$;
\item $  f_j \nearrow f$ a.e.\ implies
$\vr(f_j) \nearrow \vr(f)$;
\item $\vr(\chi_E)<\infty$ \ for every $E\subseteq R$ of finite measure;
\item  if $E\subseteq R$ is of finite measure, then $\int_{E} f\,\d\mu \le C_E
\vr(f)$, where $C_E$ is a positive constant possibly depending on $E$ and $\vr$ but not on $f$.
\end{enumerate}
If, in addition, $\vr$ satisfies
\begin{itemize}
\item[(P6)] $\vr(f) = \vr(g)$ whenever
$f \sim g$,
\end{itemize}
then we say that $\vr$ is a
\emph{rearrangement-invariant (Banach) function norm}.

If $\vr$ is a rearrangement-invariant function norm, then the set
\begin{equation*}
X=X({\vr})=\{f\in\M(R, \mu)\colon \vr(|f|)<\infty\}
\end{equation*}
equipped with the norm $\|f\|_X=\vr(|f|)$, $f\in X$, is called a~\emph{rearrangement-invariant space} (corresponding to the rearrangement-invariant function norm $\vr$). We also sometimes write $X(R, \mu)$ to stress the underlying measure space. Note that the quantity $\|f\|_{X}$ is actually well defined for every $f\in\M(R, \mu)$ and
\begin{equation*}
f\in X\quad\Leftrightarrow\quad\|f\|_X<\infty.
\end{equation*}

With any rearrangement-invariant function norm $\vr$, there is associated another functional, $\vr'$, defined for $g \in  \Mpl(R, \mu)$ as
\begin{equation*}
\vr'(g)=\sup\left\{\int_{R} fg\,\d\mu\colon f\in\Mpl(R, \mu),\ \vr(f)\leq 1\right\}.
\end{equation*}
It turns out that $\vr'$ is also a~rearrangement-invariant function norm, which is called the~\emph{associate norm} of $\vr$. Moreover, for every rearrangement-invariant function norm $\vr$ and every $f\in\Mpl(R, \mu)$, we have (see~\citep[Chapter~1, Theorem~2.9]{BS}) that
\begin{equation*}
\vr(f)=\sup\left\{\int_{R}fg\,\d\mu\colon g\in\Mpl(R, \mu),\ \vr'(f)\leq 1\right\}.
\end{equation*}
If $\vr$ is a~rearrangement-invariant function norm, $X=X({\vr})$ is the rearrangement-invariant space determined by $\vr$, and $\vr'$ is the associate norm of $\vr$, then the function space $X({\vr'})$ determined by $\vr'$ is called the \emph{associate space} of $X$ and is denoted by $X'$. We always have that
\begin{equation}\label{prel:X''}
(X')'=X,
\end{equation}
and we shall write $X''$ instead of $(X')'$. Furthermore, the \emph{H\"older inequality}
\begin{equation}\label{prel:holder}
\int_{R}fg\,\d \mu\leq\|f\|_{X}\|g\|_{X'}
\end{equation}
holds for every $f,g\in \M(R, \mu)$.

An important corollary of the Hardy--Littlewood inequality~\eqref{prel:HL} is the fact that, if $f\in M(R, \mu)$ and $X$ is a~rearrangement-invariant space over $(R, \mu)$, then we in fact have that
\begin{equation}\label{prel:assocnormwithstars}
\|f\|_{X}=\sup\left\{\int_0^{\infty}g^*(t)f^*(t)\,\d t\colon \|g\|_{X'}\leq 1\right\}.
\end{equation}

For every rearrangement-invariant space $X$ over a $\sigma$-finite, nonatomic measure space $(R, \mu)$, there is a~unique rearran\-gement-invariant space $X(0,\mu(R))$ over the interval $(0,\mu(R))$ endowed with the one-dimensional Lebesgue measure such that $\|f\|_X=\|f^*\|_{X(0,\mu(R))}$. This space is called the~\textit{representation space} of $X$. This follows from the Luxemburg representation theorem (see \citep[Chapter~2, Theorem~4.10]{BS}). Throughout this paper, the representation space of a rearrangement-invariant space $X$ will be denoted by $X(0,\mu(R))$. It is worth noting that, when $R=(0,a)$, $a\in(0,\infty]$,  and $\mu$ is the Lebesgue measure, then every $X$ over $(R, \mu)$ coincides with its representation space.

If $X$ is a rearrangement-invariant space, we define its \textit{fundamental function}, $\varphi_X$, as
\begin{equation*}
\varphi_X(t)=\|\chi_E\|_X,\ t\in[0,\mu(R)),
\end{equation*}
where $E\subseteq R$ is any set such that $\mu(E)=t$. Property (P6) of rearrangement-invariant function norms guarantees that the fundamental function is well defined. Moreover, we have that
\begin{equation}\label{prel:fundamentalfuncsidentity}
	\varphi_X(t)\varphi_{X'}(t)=t\quad\text{for every $t\in[0,\mu(R))$}.
\end{equation}
The fundamental function $\varphi_X$ is a \emph{quasiconcave function on $[0,\mu(R))$}, that is, $\varphi_X(t)=0$ if and only if $t=0$, $\varphi_X$ is nondecreasing on $[0,\mu(E))$, and the function $t\mapsto\frac{\varphi_X(t)}{t}$ is nonincreasing on $(0,|E|)$.

There are always the smallest and the largest rearrangement-invariant spaces over $(R, \mu)$ whose fundamental functions are equivalent to a given quasiconcave function $\varphi$ on $[0,\mu(R))$. More precisely, the functionals $\|\cdot\|_{\Lambda_\varphi(R, \mu)}$ and $\|\cdot\|_{M_\varphi(R, \mu)}$ defined as
\begin{equation}\label{prel:defLorentzendpoint}
\|f\|_{\Lambda_\varphi(R, \mu)}=\int_{[0,\mu(R))}f^*(t)\,\d\tilde{\varphi}(t),\ f\in\Mpl(R, \mu),
\end{equation}
where $\tilde{\varphi}$ is the least concave majorant of $\varphi$, which satisfies $\frac{1}{2}\tilde{\varphi}\leq\varphi\leq\tilde{\varphi}$ on $[0,\mu(R))$, and
\begin{equation*}
\|f\|_{M_\varphi(R, \mu)}=\sup_{t\in(0,\mu(R))}f^{**}(t)\varphi(t),\ f\in\Mpl(R, \mu),
\end{equation*}
are rearrangement-invariant function norms, and the corresponding rearrangement-invariant spaces have the following properties. We have that $\varphi_{M_\varphi(R, \mu)}=\varphi$ and $\varphi_{\Lambda_\varphi(R, \mu)}=\tilde{\varphi}$, and
\begin{equation}\label{prel:endpoints}
\Lambda_\varphi(R, \mu)\hookrightarrow X\hookrightarrow M_\varphi(R, \mu)
\end{equation}
whenever $X$ is a rearrangement-invariant space over $(R, \mu)$ whose fundamental function is equivalent to $\varphi$. The integral in \eqref{prel:defLorentzendpoint} is to be interpreted as the Lebesgue-Stieltjes integral. The spaces $M_\varphi(R, \mu)$ and $\Lambda_\varphi(R, \mu)$ are sometimes called a \emph{Marcinkiewicz endpoint space} and a \emph{Lorentz endpoint space}, respectively. For more information on endpoint spaces, we refer the reader to \citep[Chapter~7, Section~10]{FSBook}.

Let $X$ and $Y$ be rearrangement-invariant spaces over the same measure space $(R, \mu)$. We write $X\hookrightarrow Y$ to denote the fact that $X$ is (continuously) embedded into $Y$, that is, there is a positive constant $C$ such that $\|f\|_{Y}\leq C\|f\|_{X}$ for every $f\in\M(R, \mu)$. Furthermore, we have that (\citep[Chapter~1, Theorem~1.8]{BS})
\begin{align}
X \subseteq Y\quad&\text{if and only if}\quad X \hookrightarrow Y,\nonumber
\intertext{and}
X \hookrightarrow Y\quad&\text{if and only if}\quad Y' \hookrightarrow X'\label{prel:embdual}
\end{align}
with the same embedding constants. We denote by $X=Y$ the fact that $X\hookrightarrow Y$ and $Y\hookrightarrow X$ simultaneously, that is, $X$ and $Y$ are equal in the set-theoretic sense and the norms on them are equivalent to each other.

We say that a function $f\in X$ has \emph{absolutely continuous norm in $X$} if $\lim\limits_{k\to\infty}\|f\chi_{E_k}\|_{X}=0$ whenever $E_k\subseteq R$, $k\in\N$, are measurable sets such that $\lim\limits_{k\to\infty}\chi_{E_k}(x)=0$ for $\mu$-a.e.~$x\in R$. Note that $f$ has absolutely continuous norm in $X$ if and only if $\lim\limits_{a\to0^+}\|f^*\chi_{(0,a)}\|_{X(0,\mu(R))}=0$. We say that a rearrangement-invariant space $X$ has absolutely continuous norm if every function $f\in X$ has absolutely continuous norm in $X$.

Let $X$ and $Y$ be rearrangement-invariant spaces over the same $\sigma$-finite, nonatomic measure space $(R, \mu)$. We say that $X$ is \emph{almost-compactly embedded} into $Y$, and we write $X\overset{*}{\hookrightarrow}Y$, if $\lim\limits_{k\to\infty}\sup\limits_{\|f\|_{X}\leq1}\|f\chi_{E_k}\|_{Y}=0$ whenever $E_k\subseteq R$, $k\in\N$, are measurable sets such that $\lim\limits_{k\to\infty}\chi_{E_k}(x)=0$ for $\mu$-a.e.~$x\in R$. If $X\overset{*}{\hookrightarrow}Y$, then $X\hookrightarrow Y$ (\citep[Theorem~7.11.5]{FSBook}) and (\citep[(3.1)]{F-MMP:10})
\begin{equation}\label{prel:almostcompembnecessary}
\lim\limits_{t\to0^+}\frac{\varphi_Y(t)}{\varphi_X(t)}=0.
\end{equation}
Furthermore,
\begin{align}
X\overset{*}{\hookrightarrow}Y\quad&\text{if and only if}\quad X(0,\mu(R))\overset{*}{\hookrightarrow}Y(0,\mu(R)),\nonumber\\
\intertext{and}
X\overset{*}{\hookrightarrow}Y\quad&\text{if and only if}\quad Y'\overset{*}{\hookrightarrow} X'.\label{prel:almostcompactdual}
\end{align}

We refer the reader to \citep{Sla:12} for more information on almost-compact embeddings.

Textbook examples of rearrangement-invariant spaces are the standard Lebesgue spaces. The functional $\|\cdot\|_{L^p(R, \mu)}$ is defined as
\begin{equation*}
\|f\|_{L^p(R, \mu)}=\begin{cases}
\left(\int_Rf(x)^p\,\d\mu(x)\right)^\frac1{p}\quad&\text{if $p\in(0,\infty)$},\\
\esssup\limits_{x\in R}f(x)\quad&\text{if $p=\infty$},
\end{cases}
\end{equation*}
for $f\in\Mpl(R, \mu)$. The functional $\|\cdot\|_{L^p(R, \mu)}$ is a rearrangement-invariant function norm if and only if $1\leq p\leq\infty$.

An important generalization of the Lebesgue functionals is constituted by the two-parameter Lorentz functionals. Let $0<p,q\leq\infty$. We define the functional $\|\cdot\|_{L^{p,q}(R, \mu)}$  as
\begin{equation*}
\|f\|_{L^{p,q}(R, \mu)}=\left\|t^{\frac1{p}-\frac1{q}}f^*(t)\right\|_{L^q(0,\mu(R))}
\end{equation*}
for  $f \in\Mpl(R, \mu)$. Here, and in what follows, we use the convention that $\frac1{\infty}=0$. The functional $\|\cdot\|_{L^{p,q}(R, \mu)}$ is equivalent to a~rearrangement-invariant function norm if and only if $1<p<\infty$ and $1\leq q\leq\infty$, or $p=q=1$, or $p=q=\infty$. In that case, the corresponding rearrangement-invariant space is called a \emph{Lorentz space} and
\begin{equation}\label{prel:lorentzass}
(L^{p,q})'(R, \mu)=L^{p',q'}(R, \mu).
\end{equation}
Note that $L^{p,p}(R, \mu)=L^p(R, \mu)$ with the same norms.

We say that a continuous function $b\colon(0,a]\to(0,\infty)$, where $a\in(0,\infty)$, is \emph{slowly varying} on $(0,a]$ if for every $\varepsilon>0$ there is $t_0\in(0,a)$ such that the functions $t\mapsto t^\varepsilon b(t)$ and $t\mapsto t^{-\varepsilon}b(t)$ are nondecreasing and nonincreasing, respectively, on the interval $(0,t_0)$. If $b$ is a slowly-varying function on $(0,a]$, so is $b^\alpha$ for any $\alpha\in\R$.

Assume now that $\mu(R)<\infty$. Let $0<p,q\le\infty$ and $b$ be a slowly-varying function on $(0,\mu(R)]$. We define the functional $\|\cdot\|_{L^{p,q;b}(R, \mu)}$ as
\begin{equation*}
\|f\|_{L^{p,q;b}(R, \mu)}=
\left\|t^{\frac{1}{p}-\frac{1}{q}}b(t)f^*(t)\right\|_{L^q(0,\mu(R))}
\end{equation*}
for $f\in\Mpl(R, \mu)$. The functional $\|\cdot\|_{L^{p,q;b}(R, \mu)}$ is equivalent to a rearrangement-invariant function norm if and only if $q\in[1,\infty]$ and one of the following conditions holds (\citep[Theorem~3.35]{P}):
\begin{equation*}
\begin{cases}
1<p<\infty;\\
p=q=1,\ \text{$b$ is equivalent to a nonincreasing function on $(0,\mu(R)]$};\\
p=\infty,\ q<\infty,\ \int_0^{\mu(R)}t^{-1}b^q(t)\,\d t<\infty;\\
p=q=\infty,\ b\in L^\infty(0,\mu(R)).
\end{cases}
\end{equation*}
If this is the  case, the corresponding rearrangement-invariant space is called a \emph{Lorentz--Karamata space}. Note that Lebesgue and Lorentz spaces are instances of Lorentz--Karamata spaces ($b\equiv1$). Another
important subclass of Lorentz-Karamata spaces is that of (generalized) \emph{Lorentz--Zygmund spaces} $L^{p,q;\alpha,\beta}(R, \mu)$, which, in the language of Lorentz--Karamata spaces, correspond to
$b(t)=\ell(t)^\alpha\ell\ell(t)^\beta$, $\alpha,\beta\in\R$, where $\ell(t)=\log\left(e\frac{\mu(R)}{t}\right)$ and $\ell\ell(t)=\log\left(e\log\left(e\frac{\mu(R)}{t}\right)\right)$, $t\in(0,\mu(R)]$. If $\beta=0$, we usually
write $L^{p,q;\alpha}(R, \mu)$ instead of $L^{p,q;\alpha,0}(R, \mu)$. We will also occasionally need Lorentz--Zygmund spaces with more than two levels of logarithms, and we define such spaces in the obvious way. The class of Lorentz--Zygmund
spaces (more generally, that of Lorentz--Karamata spaces) encompasses not only Lebesgue spaces and Lorentz spaces, but also all types of exponential and logarithmic classes, and also the spaces discovered independently by Maz'ya (in
a~somewhat implicit form involving capacitary estimates~\citep[pp.~105 and~109]{Ma:11}), Hansson~\citep{Ha} and Brezis--Wainger~\citep{BW}, who used them for describing the sharp target space in a limiting Sobolev embedding (the
spaces can be also traced in the works of Brudnyi~\citep{B} and, in a more general setting, Cwikel and Pustylnik~\citep{CP}). For more information on Lorentz--Karamata and Lorentz--Zygmund spaces, we refer the interested reader to \citep{OP,P}.

A large number of rearrangement-invariant spaces (including the Lorentz-Karamata ones) are actually just special instances of the so-called \emph{classical Lorentz spaces} $\Lambda^q(v)$, $q\in[1, \infty]$, for suitable choices of the weight function $v$. A \emph{weight} on $(0,\mu(R))$ is any nonnegative, measurable function on $(0,\mu(R))$ that is positive on the interval $(0,\delta)$ for some $\delta\in(0,\mu(R))$ and such that $V(t)<\infty$ for every $t\in(0,\mu(R))$, where $V(t)=\int_0^tv(s)\,\d s$, $t\in(0,\mu(R))$. If $v$ is a weight on $(0,\mu(R))$, we define the functional $\|\cdot\|_{\Lambda^q(v)}$ as
\begin{equation*}
\|f\|_{\Lambda^q(v)}=\begin{cases}
\left(\int_0^{\mu(R)}f^*(t)^qv(t)\,\d t\right)^\frac1{q}\quad&\text{if $q\in[1,\infty)$},\\
\esssup\limits_{t\in(0,\mu(R))}f^*(t)v(t)\quad&\text{if $q=\infty$},
\end{cases}
\end{equation*}
for $f\in\Mpl(R, \mu)$. The functional $\|\cdot\|_{\Lambda^q(v)}$ is equivalent to a rearrangement-invariant function norm if and only if
\begin{equation}\label{prel:lambdari}
\begin{cases}
\frac1{t}\int_0^tv(u)\,\d u\lesssim\frac1{s}\int_0^sv(u)\,\d u\quad\text{for every $0<s< t<\mu(R)$}\quad&\text{if $q=1$},\\
\int_0^ts^{q'}v(s)V^{-q'}(s)\,\d s\lesssim t^{q'}V^{1-q'}(t)\quad\text{for every $0<t<\mu(R)$}\quad&\text{if $q\in(1,\infty)$},\\
\text{$\tilde{v}$ is a finite function}\quad\text{and}\quad\sup\limits_{t\in(0,\mu(R))}\tilde{v}(t)\frac1{t}\int_0^t\frac1{\tilde{v}(s)}\,\d s<\infty\quad&\text{if $q=\infty$},
\end{cases}
\end{equation}
where $\tilde{v}(t)=\esssup\limits_{s\in(0,t)}v(s)$, $t\in(0,\mu(R))$, that is, $\tilde{v}$ is the least nondecreasing (essential) majorant of $v$. We refer the reader to \citep{S:90} for $q\in(1,\infty)$, to \citep{CGS:96} for $q=1$  and to \citep{GS:14} for $q\in(1,\infty]$. The multiplicative constants in \eqref{prel:lambdari} may depend only on $q$ and $v$.

Finally, we define \emph{Sobolev-type spaces built upon rearrangement-invariant spaces}. Let $\Omega$ be a bounded Lipschitz domain (e.g.~\citep[Chapter~4, 4.9]{AF:03}) in $\rn$, $n \geq 2$. Given $m \in \N$ and a rearrangement-invariant space $X(\Omega)$ over $\Omega$ endowed with the Lebesgue measure, the $m$-th order Sobolev-type space $W^m X(\Omega)$ is defined as
\begin{equation*}
W^m X(\Omega) =   \big\{u\colon \hbox{$u$ is $m$-times weakly differentiable in $\Omega$ and $|\nabla ^k u| \in X(\Omega )$ for $k=0, \dots , m$}\big\}.
\end{equation*}
Here, $\nabla^k u$ denotes the vector of all $k$-th order weak derivatives of $u$ and $\nabla ^0 u=u$. The Sobolev-type space  $W^m X(\Omega)$ equipped with the norm
\begin{equation*}
\|u\|_{W^m X(\Omega )} = \sum _{k=0}^{m} \| \, |\nabla^k u| \, \|_{X(\Omega  )},\ u \in W^m X(\Omega),
\end{equation*}
is a Banach space. When $X(\Omega) = L^p (\Omega )$, $p \in [1,\infty]$, one has  $W^mX(\Omega) = W^{m,p}(\Omega )$ (the standard Sobolev space of $m$-th order on $\Omega$).

We say that $W^mX(\Omega)$ is embedded into a rearrangement-invariant space $Y(\overline{\Omega}, \nu)$ on $\overline{\Omega}$ endowed with a $d$-Ahlfors measure $\nu$, and we write $W^mX(\Omega)\rightarrow Y(\overline{\Omega}, \nu)$, if there is a bounded linear operator $\Tr\colon W^mX(\Omega)\to Y(\overline{\Omega}, \nu)$ such that $\Tr u=u$ for every $u\in W^mX(\Omega)\cap\mathcal C(\overline{\Omega})$. If the operator is compact, we say that the embedding is compact.

\section{On enlarging Marcinkiewicz spaces}\label{sec:q1q2}

We begin with a proposition concerning concave functions, which, while elementary, is of independent interest in the theory of Marcinkiewicz spaces.

\begin{proposition}\label{prop:piecewiseconcavefunction}
Let $\varphi\colon[0,a]\rightarrow[0,\infty)$, where $a\in(0,\infty)$, be a nondecreasing, concave function vanishing only at the origin. Assume that
\begin{equation}\label{thm:piecewiseconcavefunction:assumption1}
\lim\limits_{t\to0^+}\frac{t}{\varphi(t)}=0
\end{equation}
and
\begin{equation}\label{thm:piecewiseconcavefunction:assumption2}
\lim\limits_{t\to0^+}\varphi(t)=0.
\end{equation}
There exists a nondecreasing, concave function $\psi\colon[0,a]\rightarrow[0,\infty)$ such that $\psi\leq\varphi$ on $[0,a]$, $\psi$ vanishes only at the origin, $\liminf\limits_{t\to0^+}\frac{\psi(t)}{\varphi(t)}=0$ and $\limsup\limits_{t\to0^+}\frac{\psi(t)}{\varphi(t)}=1$. 
\end{proposition}
\begin{proof}
We shall find, by induction, two sequences $\{t_k\}_{k=1}^\infty$, $\{\tau_k\}_{k=1}^\infty$ of positive numbers converging to $0$ such that, for each $k\in\N$,
\begin{align}
&t_{k+1}<\tau_{k}<t_k,\nonumber\\
&\frac{\varphi(\tau_k)-\varphi(t_{k+1})}{\tau_k-t_{k+1}}>2^{k+1}\frac{\varphi(t_k)-\varphi(t_{k+1})}{t_k-t_{k+1}},\label{thm:piecewiseconcavefunction:property1}\\
&\varphi(t_{k+1})\leq2^{-k-1}\varphi(\tau_k)\label{thm:piecewiseconcavefunction:property2}.
\end{align}
Set $t_1=a$ and assume that we have already found $t_1,\dots,t_k$ and $\tau_1,\dots,\tau_{k-1}$ for some $k\in\N$. By \eqref{thm:piecewiseconcavefunction:assumption1} there exists $\tau_k\in(0,\frac{t_k}{2})$ such that $\frac{\varphi(\tau_k)}{\tau_k}>2^{k+1}\frac{\varphi(t_k)}{t_k}$. Since $\lim\limits_{t\to0^+}\frac{\varphi(\tau_k)-\varphi(t)}{\tau_k-t} = \frac{\varphi(\tau_k)}{\tau_k}$ by \eqref{thm:piecewiseconcavefunction:assumption2}, there exists $t_{k+1}\in(0,\tau_k)$ such that $\frac{\varphi(\tau_k) - \varphi(t_{k+1})}{\tau_k - t_{k+1}}>2^{k+1}\frac{\varphi(t_k)}{t_k}$. Moreover, we can find $t_{k+1}$ in such a way that $\varphi(t_{k+1})\leq2^{-k-1}\varphi(\tau_k)$ thanks to \eqref{thm:piecewiseconcavefunction:assumption2} and the fact that $\varphi(\tau_k)\neq0$.

Clearly $t_{k+1}<\tau_k<t_k$ and $t_{k+1}\leq\frac1{2^k}$. Since $\varphi$ is concave, we have that $\frac{\varphi(t_k)}{t_k}\geq\frac{\varphi(t_k)-\varphi(t_{k+1})}{t_k-t_{k+1}}$. Hence $\frac{\varphi(\tau_k) - \varphi(t_{k+1})}{\tau_k - t_{k+1}}>2^{k+1}\frac{\varphi(t_k)-\varphi(t_{k+1})}{t_k-t_{k+1}}$. This completes the inductive step.

We define the function $\psi\colon[0,a]\rightarrow[0,\infty)$ as
\begin{equation*}
\psi(t)=\begin{cases}\varphi(t_{k+1})+\frac{\varphi(t_k)-\varphi(t_{k+1})}{t_k-t_{k+1}}(t-t_{k+1}),\ &t\in(t_{k+1},t_k],\ k\in\N,\\
0,\ &t=0.
\end{cases}
\end{equation*}
Note that, since $\varphi$ is concave and nondecreasing, so is $\psi$. Clearly, $\psi(t_k)=\varphi(t_k)$ for each $k\in\N$. Furthermore, since $\varphi$ is concave and vanishes only at the origin, we have $0<\psi\leq\varphi$ on $(0,a]$. Hence $\limsup\limits_{t\to0^+}\frac{\psi(t)}{\varphi(t)}=1$.

Finally, for each $k\in\N$,
\begin{align*}
\frac{\psi(\tau_k)}{\varphi(\tau_k)}&=\frac{\varphi(t_{k+1})+\frac{\varphi(t_k)-\varphi(t_{k+1})}{t_k-t_{k+1}}(\tau_k-t_{k+1})}{\varphi(\tau_k)} \leq \frac{\varphi(t_{k+1})+2^{-k-1}\frac{\varphi(\tau_k)-\varphi(t_{k+1})}{\tau_k-t_{k+1}}(\tau_k-t_{k+1})}{\varphi(\tau_k)}\\
&=\frac{(1-2^{-k-1})\varphi(t_{k+1}) + 2^{-k-1}\varphi(\tau_k)}{\varphi(\tau_k)}\leq\frac{2^{-k}\varphi(\tau_k)}{\varphi(\tau_k)}=2^{-k},
\end{align*}
where the first and the second inequalities follow from \eqref{thm:piecewiseconcavefunction:property1} and \eqref{thm:piecewiseconcavefunction:property2}, respectively. Hence $\liminf\limits_{t\to0^+}\frac{\psi(t)}{\varphi(t)}=0$.
\end{proof}

We are now in the position to give a complete answer to \hyperref[q:noncompactMarc]{Question~\ref*{q:noncompactMarc}}. Note that $X(\Omega)\not\subseteq L^{\frac{n}{m},1}(\Omega)$ is equivalent to $t^{\frac{m}{n}-1}\notin X'(0,|\Omega|)$. This follows easily from \eqref{prel:embdual} combined with \eqref{prel:lorentzass}. The latter condition is usually easier to verify, and so we use it in the statement of the following theorem.

\begin{theorem}\label{T:answer-to-Q1}
Let $\Omega\subseteq\rn$ be a bounded Lipschitz domain, $m\in\N$, $m<n$, $d\in[n-m,n]$ and $\nu$ a $d$-Ahlfors measure on $\overline{\Omega}$. Assume that $X(\Omega)$ is a rearrangement-invariant space such that $t^{\frac{m}{n}-1}\notin X'(0,|\Omega|)$. Furthermore, if $d=n-m$, we assume that $X(\Omega)\neq L^1(\Omega)$. There is a rearrangement-invariant space $Y(\overline{\Omega}, \nu)$ such that $M_{Y_X}(\overline{\Omega}, \nu)\subsetneq Y(\overline{\Omega}, \nu)$ and the embedding $W^mX(\Omega)\hookrightarrow Y(\overline{\Omega}, \nu)$ is not compact.
\end{theorem}

\begin{proof}
Set $\varphi=\varphi_{Y_X}$, where $\varphi_{Y_X}$ is the fundamental function of $Y_X(\overline{\Omega}, \nu)$. Without loss of generality, we may assume that $\varphi$ is concave on $[0,\nu(\overline{\Omega})]$ (see~\citep[Chapter~2, Proposition~5.11]{BS}).
The assumption $t^{\frac{m}{n}-1}\notin X'(0,|\Omega|)$ is equivalent to $X(\Omega)\not\subseteq L^{\frac{n}{m},1}(\Omega)$, which ensures that $Y_X(\overline{\Omega}, \nu)\neq L^\infty(\overline{\Omega}, \nu)$. Hence $\varphi$
satisfies \eqref{thm:piecewiseconcavefunction:assumption2} (\citep[Theorem~5.2]{Sla:12}). Next, we claim that $Y_X(\overline{\Omega}, \nu)\neq L^1(\overline{\Omega}, \nu)$. Indeed, if $d\in(n-m,n]$, then this follows from
$Y_X(\overline{\Omega}, \nu)\subseteq Y_{L^1}(\overline{\Omega}, \nu)=L^{\frac{d}{n-m},1}(\overline{\Omega}, \nu)\subsetneq L^1(\overline{\Omega}, \nu)$ (\citep[Theorem~3.1]{CPS_OrlLor:20}). If $d=n-m$, then $X(\Omega)\neq L^1(\Omega)$ is assumed. Therefore, it follows from \citep[Theorem~5.3]{Sla:12} that $X(\Omega)\stackrel{*}{\hookrightarrow} L^1(\Omega)$. That, however, in turn implies that $Y_X(\overline{\Omega}, \nu)\neq L^1(\overline{\Omega}, \nu)$ thanks to \citep[Proposition~3.5]{CaMi:19} combined with
\citep[Theorem~3.6]{CM:20} (see also \citep[Theorem~4.6]{Sla:15}). Consequently, applying \citep[Theorem~5.3]{Sla:12} once again, we obtain that $\varphi$ satisfies
also \eqref{thm:piecewiseconcavefunction:assumption1}.

By virtue of \hyperref[prop:piecewiseconcavefunction]{Proposition~\ref*{prop:piecewiseconcavefunction}}, there is a nondecreasing, concave function $\psi\colon[0,\nu(\overline{\Omega})]\rightarrow[0,\infty)$ such that $\psi\leq\varphi$ on $[0,\nu(\overline{\Omega})]$, $\liminf\limits_{t\to0^+}\frac{\psi(t)}{\varphi(t)}=0$ and $\limsup\limits_{t\to0^+}\frac{\psi(t)}{\varphi(t)}=1$. We have that $M_{Y_X}(\overline{\Omega}, \nu)\subsetneq M_\psi(\overline{\Omega}, \nu)$ because $\psi\leq\varphi$ on $[0,\nu(\overline{\Omega})]$ and $\liminf\limits_{t\to0^+}\frac{\psi(t)}{\varphi(t)}=0$. Furthermore, $Y_X(\overline{\Omega}, \nu)$ is not almost compactly embedded into $M_\psi(\overline{\Omega}, \nu)$. This follows from the fact that $\limsup\limits_{t\to0^+}\frac{\psi(t)}{\varphi(t)}=1$ and \eqref{prel:almostcompembnecessary}. Therefore, $W^mX(\Omega)\hookrightarrow M_\psi(\overline{\Omega}, \nu)$ is not compact thanks to \citep[Theorem~4.1]{CM:20}. Hence $Y(\overline{\Omega}, \nu)=M_\psi(\overline{\Omega}, \nu)$ has the desired properties.
\end{proof}

\begin{remark}
We would like to point out that the assumption of Theorem~\ref{T:answer-to-Q1} that $X(\Omega)\neq L^1(\Omega)$ when $d=n-m$ brings no restriction at all, for if
$X(\Omega)=L^1(\Omega)$ and $d=n-m$, then
\begin{equation*}
Y_X(\overline{\Omega}, \nu)=L^{\frac{d}{n-m},1}(\overline{\Omega}, \nu)=L^1(\overline{\Omega}, \nu)=M_{L^1}(\overline{\Omega}, \nu).
\end{equation*}
The first identity follows from~\citep[Theorem~3.1]{CPS_OrlLor:20}, and the remaining ones are well known. Since
$L^1(\overline{\Omega}, \nu)$ is the largest rearrangement-invariant space over $(\overline{\Omega}, \nu)$ (\citep[Chapter~2, Corollary~6.7]{BS}), there is no rearrangement-invariant space $Y(\overline{\Omega}, \nu)$ bigger than $Y_X(\overline{\Omega}, \nu)$; thus the answer to \hyperref[q:noncompactMarc]{Question~\ref*{q:noncompactMarc}} is negative in this case.
\end{remark}

Having answered \hyperref[q:noncompactMarc]{Question~\ref*{q:noncompactMarc}}, we now turn our attention to \hyperref[q:noncompactfundamentallybigger1]{Question~\ref*{q:noncompactfundamentallybigger1}}. We start with some observations.
\begin{proposition}
Let $\Omega\subseteq\rn$ be a bounded Lipschitz domain, $m\in\N$, $m<n$, $d\in[n-m,n]$ and $\nu$ a $d$-Ahlfors measure on $\overline{\Omega}$. If $Y(\overline{\Omega}, \nu)$ is a rearrangement-invariant space satisfying properties required by \hyperref[q:noncompactfundamentallybigger1]{Question~\ref*{q:noncompactfundamentallybigger1}}, then:
\begin{itemize}
\item $Y(\overline{\Omega}, \nu)$ is not a Marcinkiewicz space;
\item $M_{Y_X}(\overline{\Omega}, \nu)\not\hookrightarrow \Lambda_{Y}(\overline{\Omega}, \nu)$, where $\Lambda_{Y}(\overline{\Omega}, \nu)$ is the Lorentz endpoint space with the same fundamental function as $Y(\overline{\Omega}, \nu)$;
\item $Y(\overline{\Omega}, \nu)$ does not have absolutely continuous norm.
\end{itemize}
\end{proposition}
\begin{proof}
If $Y(\overline{\Omega}, \nu)$ were a Marcinkiewicz space, we would get $M_{Y_X}(\overline{\Omega}, \nu)\stackrel{*}{\hookrightarrow}Y(\overline{\Omega}, \nu)$ by~\cite[Corollary~7.5]{Sla:12} and, consequently, $W^{m}X(\Omega)\hookrightarrow\hookrightarrow Y(\overline{\Omega}, \nu)$ due to~\cite[Theorem~4.1]{CM:20}.

If $M_{Y_X}(\overline{\Omega}, \nu)\hookrightarrow \Lambda_{Y}(\overline{\Omega}, \nu)$, then \cite[Corollary~7.3]{Sla:12} would imply that $M_{Y_X}(\overline{\Omega}, \nu)\stackrel{*}{\hookrightarrow}\Lambda_{Y}(\overline{\Omega}, \nu)$. If $Y(\overline{\Omega}, \nu)$ had absolutely continuous norm, then $M_{Y_X}(\overline{\Omega}, \nu)\hookrightarrow Y(\overline{\Omega}, \nu)$ coupled with~\cite[Theorem~7.2]{Sla:12} would imply that $M_{Y_X}(\overline{\Omega}, \nu)\stackrel{*}{\hookrightarrow}\Lambda_{Y}(\overline{\Omega}, \nu)$. Either way, we would obtain from $M_{Y_X}(\overline{\Omega}, \nu)\stackrel{*}{\hookrightarrow}\Lambda_{Y}(\overline{\Omega}, \nu)$ that $W^{m}X(\Omega)\hookrightarrow\hookrightarrow\Lambda_{Y}(\overline{\Omega}, \nu)\hookrightarrow Y(\overline{\Omega}, \nu)$ thanks to \cite[Theorem~4.1]{CM:20} and \eqref{prel:endpoints}.
\end{proof}

While the preceding proposition limits what a potential space $Y(\overline{\Omega}, \nu)$ sought in \hyperref[q:noncompactfundamentallybigger1]{Question~\ref*{q:noncompactfundamentallybigger1}} can be,  there is also a natural restriction on the space $Y_X(\overline{\Omega}, \nu)$ should the answer to \hyperref[q:noncompactfundamentallybigger1]{Question~\ref*{q:noncompactfundamentallybigger1}} be positive. It turns out that $Y_X(\overline{\Omega}, \nu)$ cannot be a Lorentz endpoint space. Indeed, suppose that $Y_X(\overline{\Omega}, \nu)=\Lambda_{Y_X}(\overline{\Omega}, \nu)$ and that $Y(\overline{\Omega}, \nu)$ is any rearrangement-invariant space satisfying $\lim\limits_{t\to0^+}\frac{\varphi_Y(t)}{\varphi_{Y_X}(t)}=0$. Then $\Lambda_{Y_X}(\overline{\Omega}, \nu)\overset{*}{\hookrightarrow}\Lambda_{Y}(\overline{\Omega}, \nu)\hookrightarrow Y(\overline{\Omega}, \nu)$ thanks to \cite[Corollary~7.5]{Sla:12} and \eqref{prel:endpoints}; hence $W^{m}X(\Omega)\hookrightarrow\hookrightarrow\Lambda_{Y}(\overline{\Omega}, \nu)\hookrightarrow Y(\overline{\Omega}, \nu)$ by virtue of \cite[Theorem~4.1]{CM:20}.

Nevertheless, if $Y_X(\overline{\Omega}, \nu)=M_{Y_X}(\overline{\Omega}, \nu)$ is a Marcinkiewicz space, we can (at least for a lot of customary choices of $X(\Omega)$) explicitly construct a space $Y(\overline{\Omega}, \nu)$ giving an affirmative answer to \hyperref[q:noncompactfundamentallybigger1]{Question~\ref*{q:noncompactfundamentallybigger1}}. Before doing that, we prove a general theorem, which provides a guideline on how to construct such spaces.

\begin{theorem}\label{thm:partanswto:noncompactfundamentallybigger1}
Assume that $\varphi\colon[0,a]\rightarrow[0,\infty)$, where $a\in(0,\infty)$, is a quasiconcave function satisfying
\begin{equation}
\frac1{t}\int_0^t\frac1{\varphi(s)}\,\d s\lesssim\frac1{\varphi(t)}\quad\text{for every $t\in(0,a)$}.\label{thm:partanswto:noncompactfundamentallybigger1:conB}
\end{equation}
Let $\tau$ be a positive, measurable function on $(0,a]$ such that
\begin{align}
\int_0^t\frac{\varphi(s)}{\tau(s)}\,\d s&\lesssim\varphi(t)\quad\text{for every $t\in(0,a)$},\label{thm:partanswto:noncompactfundamentallybigger1:conintphiovertau}\\
\int_t^a\frac1{\tau(s)}\,\d s&<\infty\quad\text{for every $t\in(0,a)$},\label{thm:partanswto:noncompactfundamentallybigger1:confunbfinite}\\
\int_0^a\frac1{\tau(s)}\,\d s&=\infty.\label{thm:partanswto:noncompactfundamentallybigger1:confunbblowsup}
\end{align}
Furthermore, assume that the function $b\colon(0,a]\to(0,\infty)$ defined as
\begin{equation*}
b(t)=1+\int_t^a\frac1{\tau(s)}\,\d s,\ t\in(0,a],
\end{equation*}
is slowly varying on $(0,a]$, and that
\begin{equation}\label{thm:partanswto:noncompactfundamentallybigger1:interepreoftauoverphi}
\frac{\tau(t)}{\varphi(t)}\approx\int_0^t\xi(s)\,\d s,\quad\text{for every $t\in(0,a)$},
\end{equation}
where $\xi$ is some positive, continuous function on $(0,a]$ that is equivalent to a nonincreasing function. Finally, assume that
\begin{equation}\label{thm:partanswto:noncompactfundamentallybigger1:supremumopbounded}
\frac1{\xi(t)}\int_0^t\xi(s)b(s)\,\d s\lesssim\int_0^tb(s)\,\d s\quad\text{for every $t\in(0,a)$}.
\end{equation}

Let $(R, \mu)$ be a nonatomic measure space such that $\mu(R)=a$. Set
\begin{equation}\label{thm:partanswto:noncompactfundamentallybigger1:normY}
\varrho_Y(f)=\sup\limits_{t\in(0,a)}\frac1{b(t)}\int_t^af^*(s)\frac{\varphi(s)}{\tau(s)}\,\d s,\ f\in\Mpl(R, \mu).
\end{equation}
The functional $\varrho_Y$ is equivalent to a rearrangement-invariant function norm on $(R, \mu)$. Let $Y(R, \mu)$ denote the corresponding rearrangement-invariant space. We have that $M_\varphi(R, \mu)\hookrightarrow Y(R, \mu)$ and $\lim\limits_{t\to0^+}\frac{\varphi_Y(t)}{\varphi(t)}=0$, but the embedding is not almost compact.
\end{theorem}
\begin{proof}
We start off by showing that the functional $\varrho_Y$, defined by \eqref{thm:partanswto:noncompactfundamentallybigger1:normY}, is equivalent to a rearrangement-invariant function norm. We claim that $\varrho_Y$ is equivalent to a functional $\sigma$ defined as
\begin{equation}\label{thm:partanswto:noncompactfundamentallybigger1:defsigma}
\sigma(f)=\sup_{\|g\|_{L^{1,1;b}(0,a)}\leq1}\int_0^af^*(t)\frac1{\int_0^t\xi(u)\,\d u}\int_0^t\xi(s)\sup_{u\in[s,a)}\frac1{\xi(u)}g^*(u)\,\d s\,\d t,\ f\in\Mpl(R, \mu),
\end{equation}
and that $\sigma$ is a rearrangement-invariant function norm. As for the equivalence, it follows from \citep[Theorem~3.33]{P} and the Hardy--Littlewood inequality \eqref{prel:HL} that
\begin{equation*}
\varrho_Y(f)=\left\|\int_t^af^*(s)\frac{\varphi(s)}{\tau(s)}\,\d s\right\|_{L^{\infty,\infty;b^{-1}}(0,a)}\approx\sup\limits_{\|g\|_{L^{1,1;b}(0,a)}\leq1}\int_0^ag^*(t)\int_t^af^*(s)\frac{\varphi(s)}{\tau(s)}\,\d s\,\d t,
\end{equation*}
where $b^{-1}=\frac1{b}$. Hence, thanks to Fubini's theorem and \eqref{thm:partanswto:noncompactfundamentallybigger1:interepreoftauoverphi},
\begin{equation}\label{thm:partanswto:noncompactfundamentallybigger1:dualexprrho}
\begin{aligned}
\varrho_Y(f)&\approx\sup\limits_{\|g\|_{L^{1,1;b}(0,a)}\leq1}\int_0^af^*(s)\frac{\varphi(s)}{\tau(s)}\int_0^sg^*(t)\,\d t\,\d s\\
&\approx\sup\limits_{\|g\|_{L^{1,1;b}(0,a)}\leq1}\int_0^af^*(s)\frac1{\int_0^s\xi(u)\,\d u}\int_0^sg^*(t)\,\d t\,\d s.
\end{aligned}
\end{equation}
It plainly follows from \eqref{thm:partanswto:noncompactfundamentallybigger1:defsigma} and \eqref{thm:partanswto:noncompactfundamentallybigger1:dualexprrho} that
\begin{equation}\label{thm:partanswto:noncompactfundamentallybigger1:rholesssigma}
\varrho_Y(f)\lesssim\sigma(f).
\end{equation}
As for the opposite inequality, we need to introduce the supremum operator $T_\xi$ defined as, for every fixed $h\in\M(0,a)$,
\begin{equation*}
T_\xi h(t)=\xi(t)\sup_{u\in[t,a)}\frac1{\xi(u)}h^*(u),\ t\in(0,a).
\end{equation*}
Note that, for each $h\in\M(0,a)$, $T_\xi h$ is equivalent to a nonincreasing function on $(0,a)$. Furthermore, observe that assumption \eqref{thm:partanswto:noncompactfundamentallybigger1:supremumopbounded} guarantees that the operator is bounded on the rearrangement-invariant space $L^{1,1;b}(0,a)$ (cf.~\citep[Theorem~3.35]{P}) owing to \citep[Theorem~3.2]{GOP}. Hence, thanks to \eqref{thm:partanswto:noncompactfundamentallybigger1:defsigma}, Fubini's theorem, \citep[Theorem~3.32]{P} combined with H\"older's inequality \eqref{prel:holder}, and \eqref{thm:partanswto:noncompactfundamentallybigger1:interepreoftauoverphi},
\begin{equation}\label{thm:partanswto:noncompactfundamentallybigger1:sigmalessrho}
\begin{aligned}
\sigma(f)&=\sup_{\|g\|_{L^{1,1;b}(0,a)}\leq1}\int_0^af^*(s)\frac1{\int_0^s\xi(u)\,\d u}\int_0^sT_\xi g(t)\,\d t\,\d s\\
&=\sup_{\|g\|_{L^{1,1;b}(0,a)}\leq1}\int_0^aT_\xi g(t)\int_t^af^*(s)\frac1{\int_0^s\xi(u)\,\d u}\,\d s\,\d t\\
&\approx\sup_{\|g\|_{L^{1,1;b}(0,a)}\leq1}\int_0^aT_\xi g(t)\int_t^af^*(s)\frac{\varphi(s)}{\tau(s)}\,\d s\,\d t\\
&\lesssim\left\|\int_t^af^*(s)\frac{\varphi(s)}{\tau(s)}\,\d s\right\|_{L^{\infty,\infty;b^{-1}}(0,a)}\sup_{\|g\|_{L^{1,1;b}(0,a)}\leq1}\|T_\xi g\|_{L^{1,1;b}(0,a)}\\
&\lesssim\varrho_Y(f).
\end{aligned}
\end{equation}
Combining \eqref{thm:partanswto:noncompactfundamentallybigger1:rholesssigma} and \eqref{thm:partanswto:noncompactfundamentallybigger1:sigmalessrho}, we obtain the desired equivalence of $\varrho_Y$ and $\sigma$. It remains to show that the functional $\sigma$ is a rearrangement-invariant function norm. We shall only prove that $\sigma$ is subadditive and that
\begin{equation}\label{thm:partanswto:noncompactfundamentallybigger1:sigmalocembintol1}
\int_R f(x)\,\d\mu(x)\lesssim\sigma(f)\quad\text{for every $f\in\Mpl(R, \mu)$}
\end{equation}
because it can be readily verified that $\sigma$ possesses all of the other properties of a rearrangement-invariant function norm. Let $f_1,f_2\in\Mpl(R, \mu)$ be given. Being the integral mean of the nonincreasing function $(0,a)\ni s\mapsto \sup\limits_{u\in[s,a)}\frac1{\xi(u)}g^*(u)$, where $g\in L^{1,1;b}(0,a)$ is a fixed function, over the interval $(0,t)$ with respect to the measure $\xi(s)\,\d s$, the function $(0,a)\ni t\mapsto \frac1{\int_0^t\xi(u)\,\d u}\int_0^tT_\xi g(s)\,\d s$ is nonincreasing on $(0,a)$, and so, thanks to Hardy's lemma \eqref{prel:hardy-lemma}, we have that
\begin{align*}
\sigma(f_1+f_2)&=\sup_{\|g\|_{L^{1,1;b}(0,a)}\leq1}\int_0^a\left(f_1+f_2\right)^*(t)\frac1{\int_0^t\xi(u)\,\d u}\int_0^tT_\xi g(s)\,\d s\,\d t\\
&\leq\sup_{\|g\|_{L^{1,1;b}(0,a)}\leq1}\int_0^af_1^*(t)\frac1{\int_0^t\xi(u)\,\d u}\int_0^tT_\xi g(s)\,\d s\,\d t\\
&\quad+ \sup_{\|g\|_{L^{1,1;b}(0,a)}\leq1}\int_0^af_2^*(t)\frac1{\int_0^t\xi(u)\,\d u}\int_0^tT_\xi g(s)\,\d s\,\d t\\
&=\sigma(f_1)+\sigma(f_2).
\end{align*}
Hence $\sigma$ is subadditive. Let $f\in\Mpl(R, \mu)$ be given. Exploiting again the fact that the function $(0,a)\ni t\mapsto \frac1{\int_0^t\xi(u)\,\d u}\int_0^tT_\xi g(s)\,\d s$ is, for every fixed $g\in\M(0,a)$, nonincreasing on $(0,a)$, we have that
\begin{align*}
\sigma(f)&=\sup_{\|g\|_{L^{1,1;b}(0,a)}\leq1}\int_0^af^*(t)\frac1{\int_0^t\xi(u)\,\d u}\int_0^tT_\xi g(s)\,\d s\,\d t\\
&\gtrsim\sup_{\|g\|_{L^{1,1;b}(0,a)}\leq1}\frac1{\int_0^a\xi(u)\,\d u}\int_0^aT_\xi g(s)\,\d s\int_0^af^*(t)\,\d t\\
&\geq\frac1{\|\chi_{(0,a)}\|_{L^{1,1;b}(0,a)}}\frac1{\int_0^a\xi(u)\,\d u}\int_0^aT_\xi \chi_{(0,a)}(s)\,\d s\int_0^af^*(t)\,\d t\\
&\approx\frac1{\|\chi_{(0,a)}\|_{L^{1,1;b}(0,a)}}\frac1{\int_0^a\xi(u)\,\d u}\int_0^a\frac{\xi(s)}{\xi(a)}\,\d s\int_0^af^*(t)\,\d t\\
&=\frac1{\|\chi_{(0,a)}\|_{L^{1,1;b}(0,a)}}\frac1{\xi(a)}\int_0^af^*(t)\,\d t\\
&\geq\frac1{\|\chi_{(0,a)}\|_{L^{1,1;b}(0,a)}}\frac1{\xi(a)}\int_Rf(x)\,\d\mu(x),
\end{align*}
where the last inequality is true thanks to Hardy-Littlewood inequality \eqref{prel:HL}. Hence \eqref{thm:partanswto:noncompactfundamentallybigger1:sigmalocembintol1} holds.

Now that we know that $Y(R, \mu)$ is equivalent to a rearrangement-invariant space, we turn our attention to its relation to the Marcinkiewicz space $M_\varphi(R, \mu)$. Note that $M_\varphi(R, \mu)\hookrightarrow Y(R, \mu)$. Indeed, since $\varphi$ satisfies \eqref{thm:partanswto:noncompactfundamentallybigger1:conB}, we have that
\begin{equation}\label{thm:partanswto:noncompactfundamentallybigger1:eq1}
\|f\|_{M_\varphi(R, \mu)}\approx\sup\limits_{t\in(0,a)}f^*(t)\varphi(t)\quad\text{for every $f\in\M(R, \mu)$}
\end{equation}
(see \cite[Lemma~2.1]{MusOl:19}); therefore, it is sufficient to verify that $\frac1{\varphi}\in Y(0,a)$, which is easy (recall \eqref{thm:partanswto:noncompactfundamentallybigger1:confunbfinite} and \eqref{thm:partanswto:noncompactfundamentallybigger1:confunbblowsup}).

Next, we claim that
\begin{equation*}
\varphi_Y(t)\lesssim\frac{\varphi(t)}{b(t)}\quad\text{for every $t\in(0,a)$}.
\end{equation*}
Indeed, thanks to \eqref{thm:partanswto:noncompactfundamentallybigger1:conintphiovertau} and the fact that the function $b$ is nonincreasing,
\begin{equation*}
\varphi_Y(t)\approx\sup\limits_{u\in(0,t]}\frac1{b(u)}\int_u^t\frac{\varphi(s)}{\tau(s)}\,\d s\leq\frac1{b(t)}\int_0^t\frac{\varphi(s)}{\tau(s)}\,\d s\lesssim\frac{\varphi(t)}{b(t)}.
\end{equation*}
Hence $\lim\limits_{t\to0^+}\frac{\varphi_Y(t)}{\varphi(t)}=0$ thanks to \eqref{thm:partanswto:noncompactfundamentallybigger1:confunbblowsup}.

Finally, in order to prove that the embedding $M_\varphi(R, \mu)\hookrightarrow Y(R, \mu)$ is not almost compact, we consider functions $f_k\in \Mpl(R, \mu)$, $k\in\N$, such that
\begin{equation*}
f^*_k(t)=\frac1{\varphi(\frac1{k})}\chi_{\left(0,\frac1{k}\right)}(t) + \frac1{\varphi(t)}\chi_{\left[\frac1{k},a\right)}(t),\ t\in(0,a).
\end{equation*}
Note that the set $M=\{f_k\colon k\in\N\}$ is bounded in $M_\varphi(R, \mu)$, for, thanks to \eqref{thm:partanswto:noncompactfundamentallybigger1:eq1},
\begin{equation*}
\|f_k\|_{M_\varphi(R, \mu)}\approx\max\left\{\frac1{\varphi(\frac1{k})}\sup_{t\in(0,\frac1{k})}\varphi(t), \sup_{t\in(\frac1{k},a)}\frac1{\varphi(t)}\varphi(t)\right\}=1,
\end{equation*}
where the equivalence constants are independent of $k$. However, $M$ does not have uniformly absolutely continuous norm in $Y(R, \mu)$. Indeed, let $\delta\in(0,a)$. Owing to \eqref{thm:partanswto:noncompactfundamentallybigger1:confunbblowsup} and \eqref{thm:partanswto:noncompactfundamentallybigger1:confunbfinite}, we can find $k\in\N$ large enough that
\begin{align}
\frac1{k}&<\delta,\nonumber\\
b\left(\frac1{k}\right)&\leq2\int_{\frac1{k}}^a\frac1{\tau(s)}\,\d s,\label{thm:partanswto:noncompactfundamentallybigger1:eq3}\\
\int_{\frac1{k}}^a\frac1{\tau(s)}\,\d s&\geq2\int_\delta^a\frac1{\tau(s)}\,\d s.\label{thm:partanswto:noncompactfundamentallybigger1:eq4}
\end{align}
Hence
\begin{equation}\label{thm:partanswto:noncompactfundamentallybigger1:eq5}
\|f^*_k\chi_{(0,\delta)}\|_{Y(0,a)}\gtrsim\sup_{t\in[\frac1{k},\delta]}\frac1{b(t)}\int_t^\delta\frac1{\tau(s)}\,\d s\geq\frac1{b\left(\frac1{k}\right)}\int_\frac1{k}^\delta\frac1{\tau(s)}\,\d s\geq\frac{1}{2}\frac{\int_\frac1{k}^\delta\frac1{\tau(s)}\,\d s}{\int_\frac1{k}^a\frac1{\tau(s)}\,\d s}\geq\frac{1}{4},
\end{equation}
where the third and last inequality is valid thanks to \eqref{thm:partanswto:noncompactfundamentallybigger1:eq3} and \eqref{thm:partanswto:noncompactfundamentallybigger1:eq4}, respectively. Since $\delta\in(0,a)$ was arbitrary, \eqref{thm:partanswto:noncompactfundamentallybigger1:eq5} implies that $M$ does not have uniformly absolutely continuous norm in $Y(R, \mu)$.
\end{proof}

\begin{remark}\label{rem:partanswto:noncompactfundamentallybigger1}
Since the question of whether a rearrangement-invariant space is (almost compactly) embedded into another one is invariant with respect to equivalently renorming the spaces, \hyperref[thm:partanswto:noncompactfundamentallybigger1]{Theorem~\ref*{thm:partanswto:noncompactfundamentallybigger1}} can be used even when the function $\varphi$ is merely equivalent to a quasiconcave function on $[0,a]$.

Furthermore, since \hyperref[thm:partanswto:noncompactfundamentallybigger1]{Theorem~\ref*{thm:partanswto:noncompactfundamentallybigger1}} has a large number of assumptions, it is worth noting some concrete, important examples of functions that satisfy the assumptions.
\begin{enumerate}[(i)]
\item If $\varphi(t)=t^\alpha\ell(t)^\beta$ with $\alpha\in(0,1)$ and $\beta\in\R$, then we can take $\tau(t)=t$, $b(t)=\ell(t)$, and $\xi(t)=t^{-\alpha}\ell(t)^{-\beta}$.
\item If $\varphi(t)=\ell(t)^\beta$ with $\beta<0$, then we can take $\tau(t)=t\ell(t)$, $b(t)=\ell\ell(t)$, and $\xi(t)=\ell(t)^{1-\beta}$.
\end{enumerate}
In both cases (i) and (ii), the function $\varphi$ is equivalent to a quasiconcave function on $[0,a]$, and $\varphi$ together with the functions $\tau$, $b$ and $\xi$, defined above, satisfies the assumptions of \hyperref[thm:partanswto:noncompactfundamentallybigger1]{Theorem~\ref*{thm:partanswto:noncompactfundamentallybigger1}}. These examples actually illustrate how to use \hyperref[thm:partanswto:noncompactfundamentallybigger1]{Theorem~\ref*{thm:partanswto:noncompactfundamentallybigger1}} when $\varphi$ has the form $\varphi(t)=t^\alpha b(t)$, where $\alpha\in[0,1)$ and $b$ is a slowly-varying function on $(0,a]$, which is the case in many customary situations.
\end{remark}

We now use \hyperref[thm:partanswto:noncompactfundamentallybigger1]{Theorem~\ref*{thm:partanswto:noncompactfundamentallybigger1}} together with \hyperref[rem:partanswto:noncompactfundamentallybigger1]{Remark~\ref*{rem:partanswto:noncompactfundamentallybigger1}} to obtain examples of spaces $Y(\overline{\Omega}, \nu)$ giving a positive answer to \hyperref[q:noncompactfundamentallybigger1]{Question~\ref*{q:noncompactfundamentallybigger1}} in the case where $X(\Omega)$ is a weak Lorentz-Zygmund space. Our approach also outlines a general way that can be used to construct a space $Y(\overline{\Omega}, \nu)$ sought in \hyperref[q:noncompactfundamentallybigger1]{Question~\ref*{q:noncompactfundamentallybigger1}} in the case where $Y_X(\overline{\Omega}, \nu)$ is a Marcinkiewicz space.
\begin{theorem}
Let $\Omega\subseteq\rn$ be a bounded Lipschitz domain, $m\in\N$, $m<n$, $d\in[n-m,n]$ and $\nu$ a $d$-Ahlfors measure on $\overline{\Omega}$. Assume that either $p\in(1,\frac{n}{m})$ and $\alpha\in\R$ or $p=\frac{n}{m}$ and $\alpha\leq1$. There is a rearrangement-invariant space $Y(\overline{\Omega}, \nu)$, whose norm is induced by the function norm defined by \eqref{thm:partanswto:noncompactfundamentallybigger1:defsigma} (and equivalent to \eqref{thm:partanswto:noncompactfundamentallybigger1:normY}) with $a=\nu\left(\overline{\Omega}\right)$,
\begin{align*}
\varphi(t)&\approx\begin{cases}
t^\frac{n-mp}{dp}\ell(t)\quad&\text{if $p\in(1,\frac{n}{m})$},\\
\ell(t)^{\alpha-1}\quad&\text{if $p=\frac{n}{m}$ and $\alpha<1$},\\
\ell\ell(t)^{-1}\quad&\text{if $p=\frac{n}{m}$ and $\alpha=1$},
\end{cases}\\
\tau(t)&=\begin{cases}
t\quad&\text{if $p\in(1,\frac{n}{m})$},\\
t\ell(t)\quad&\text{if $p=\frac{n}{m}$ and $\alpha<1$},\\
t\ell\ell(t)\quad&\text{if $p=\frac{n}{m}$ and $\alpha=1$},
\end{cases}\\
\intertext{and}
b(t)&=\begin{cases}
\ell(t)\quad&\text{if $p\in(1,\frac{n}{m})$},\\
\ell\ell(t)\quad&\text{if $p=\frac{n}{m}$ and $\alpha<1$},\\
\ell\ell\ell(t)\quad&\text{if $p=\frac{n}{m}$ and $\alpha=1$},
\end{cases}
\end{align*}
such that
\begin{itemize}
\item $W^mL^{p,\infty;\alpha}(\Omega)\hookrightarrow Y(\overline{\Omega}, \nu)$ non-compactly,
\item $M_{Y_{L^{p,\infty;\alpha}}}(\overline{\Omega}, \nu)\subsetneq Y(\overline{\Omega}, \nu)$,
\item $\lim\limits_{t\to0^+}\frac{\varphi_Y(t)}{\varphi_{Y_{L^{p,\infty;\alpha}}}(t)}=0$.
\end{itemize}
\end{theorem}
\begin{proof}
It is known (see \cite[Theorem~5.1]{CM:20}) that
\begin{equation*}
Y_{L^{p,\infty;\alpha}}(\overline{\Omega}, \nu)=\begin{cases}
L^{\frac{dp}{n-mp},\infty;\alpha}(\overline{\Omega}, \nu)\quad&\text{if $p\in(1,\frac{n}{m})$},\\
L^{\infty, \infty;\alpha-1}(\overline{\Omega}, \nu)\quad&\text{if $p=\frac{n}{m}$ and $\alpha<1$},\\
L^{\infty, \infty;0,-1}(\overline{\Omega}, \nu)\quad&\text{if $p=\frac{n}{m}$ and $\alpha=1$}.
\end{cases}
\end{equation*}
Furthermore, $Y_{L^{p,\infty;\alpha}}(\overline{\Omega}, \nu)$ is equivalent to a Marcinkiewicz space $M_\psi(\overline{\Omega}, \nu)$ with $\psi\approx \varphi$. The claim now follows from \hyperref[thm:partanswto:noncompactfundamentallybigger1]{Theorem~\ref*{thm:partanswto:noncompactfundamentallybigger1}} combined with \citep[Theorem~4.1]{CM:20}.
\end{proof}

\section{On enlarging general rearrangement-invariant spaces}\label{prel:q3}
The following proposition provides a general principle that can be exploited to construct a space $Y(\overline{\Omega}, \nu)$ sought in \hyperref[q:noncompactfundamentallybigger2]{Question~\ref*{q:noncompactfundamentallybigger2}}.
\begin{proposition}\label{prop:paransq2}
Let $\left(R, \mu\right)$ be a finite, nonatomic measure space. Let $Z_1(R, \mu)$ and $Z_2(R, \mu)$ be rearrangement-invariant spaces. Assume that $\lim\limits_{t\to0^+}\frac{\varphi_{Z_2}(t)}{\varphi_{Z_1}(t)}=0$ and $Z_1(R, \mu)\not\subseteq Z_2(R, \mu)$. Set $Y(R, \mu)=\left(Z_1'(R, \mu)\cap Z_2'(R, \mu)\right)'$. We have that $\lim\limits_{t\to0^+}\frac{\varphi_{Y}(t)}{\varphi_{Z_1}(t)}=0$ and $Z_1(R, \mu)\hookrightarrow Y(R, \mu)$, but the embedding is not almost compact.

Moreover, if either of the spaces $Z_1'(R, \mu)$ and $Z_2'(R, \mu)$ has absolutely continuous norm, then $Y(R, \mu)=Z_1(R, \mu)+Z_2(R, \mu)$.
\end{proposition}
\begin{proof}
Combining \eqref{prel:X''} and \eqref{prel:embdual}, we obtain that $Z_1(R, \mu)\hookrightarrow Y(R, \mu)$. Next, we have that $\varphi_{Y'}=\max\{\varphi_{Z_1'}, \varphi_{Z_2'}\}$, whence
\begin{equation*}
\lim_{t\to0^+}\frac{\varphi_Y(t)}{\varphi_{Z_1}(t)}=\lim_{t\to0^+}\frac{\varphi_{Z_1'}(t)}{\varphi_{Y'}(t)}\leq\lim_{t\to0^+}\frac{\varphi_{Z_1'}(t)}{\varphi_{Z_2'}(t)}=\lim_{t\to0^+}\frac{\varphi_{Z_2}(t)}{\varphi_{Z_1}(t)}=0
\end{equation*}
thanks to \eqref{prel:fundamentalfuncsidentity}.

It remains to show that $Z_1(R, \mu)$ is not almost-compactly embedded into $Y(R, \mu)$, which is equivalent to showing that $Z_1'(R, \mu)\cap Z_2'(R, \mu)$ is not almost-compactly embedded into $Z_1'(R, \mu)$ thanks to \eqref{prel:almostcompactdual}. Owing to \eqref{prel:embdual}, we have that $Z_2'(R, \mu)\not\subseteq Z_1'(R, \mu)$, and so the embedding $Z_1'(R, \mu)\cap Z_2'(R, \mu)\hookrightarrow Z_1'(R, \mu)$ is not almost compact by \citep[Lemma~3.7]{F-MMP:10}.

Finally, if either of the spaces $Z_1'(R, \mu)$ and $Z_2'(R, \mu)$ has absolutely continuous norm, then $\left(Z_1'(R, \mu)\cap Z_2'(R, \mu)\right)'=Z_1(R, \mu)+Z_2(R, \mu)$ thanks to \citep[Chapter~3, Exercise~5]{BS}.
\end{proof}

The following two theorems suggest how to find a space $Z_2(R, \mu)$ for a given space $Z_1(R, \mu)=\Lambda^q(v)$ in such a way that \hyperref[prop:paransq2]{Proposition~\ref*{prop:paransq2}} can be used. Since a large number of customary function spaces can be described as a $\Lambda^q(v)$ space, the theorems provide quite general tools for producing spaces $Y(\overline{\Omega}, \nu)$ sought in \hyperref[q:noncompactfundamentallybigger2]{Question~\ref*{q:noncompactfundamentallybigger2}}.
\begin{theorem}\label{thm:paransq2lambdaspaces}
Let $\left(R, \mu\right)$ be a finite, nonatomic measure space. Let $q\in(1,\infty)$ and $v$ be a weight on $(0,\mu(R))$ satisfying
\begin{equation*}
\int_0^ts^{q'}v(s)V(s)^{-q'}\,\d s\lesssim t^{q'}V(t)^{1-q'}\quad\text{for every $t\in(0,\mu(R))$}.
\end{equation*}
Set $Z(R, \mu)=\Lambda^q(v)$. Let $r\in[1,q)$ and $w\colon(0,\mu(R))\to(0,\infty)$ be a weight satisfying \eqref{prel:lambdari} with $v$ and $q$ replaced by $w$ and $r$, respectively,

\begin{equation}\label{thm:paransq2lambdaspaces:ratiooffundamentalfuncs}
    \lim\limits_{t\to0^+}\frac{W(t)^\frac1{r}}{V(t)^\frac1{q}}=0
\end{equation}
and
\begin{equation}\label{thm:paransq2lambdaspaces:XnotintoZ}
    \int_0^{\mu(R)}\left(\frac{W(t)}{V(t)}\right)^\frac{q}{q-r}v(t)\,\d t=\infty.
\end{equation}
Set $Y(R, \mu)=\Lambda^q(v) + \Lambda^r(w)$. We have that $\lim\limits_{t\to0^+}\frac{\varphi_{Y}(t)}{\varphi_{Z}(t)}=0$ and $Z(R, \mu)\hookrightarrow Y(R, \mu)$, but the embedding is not almost compact.
\end{theorem}

\begin{proof}
The assertion of the theorem will immediately follow from \hyperref[prop:paransq2]{Proposition~\ref*{prop:paransq2}} with $Z_1(R, \mu)=\Lambda^q(v)$ and $Z_2(R, \mu)=\Lambda^r(w)$ once we verify its assumptions. Note that $\Lambda^q(v)$ and $\Lambda^r(w)$ are equivalent to rearrangement-invariant spaces. Assumption \eqref{thm:paransq2lambdaspaces:XnotintoZ} is equivalent to the fact that $Z_1(R, \mu)\not\subseteq Z_2(R, \mu)$ by \citep[Proposition~1]{S:93}. Since $\varphi_{Z_1}\approx V^\frac1{q}$ and $\varphi_{Z_2}\approx W^\frac1{r}$, we plainly have that $\lim\limits_{t\to0^+}\frac{\varphi_{Z_2}(t)}{\varphi_{Z_1}(t)}=0$ thanks to \eqref{thm:paransq2lambdaspaces:ratiooffundamentalfuncs}. It only remains to observe that $Z_1'(R, \mu)$ has absolutely continuous norm. Indeed, owing to \citep[Theorem~1]{S:90} combined with \eqref{prel:assocnormwithstars}, we have that
\begin{align*}
\|g^*\chi_{(0,a)}\|_{Z_1'(0,\mu(R))}&\approx\left(\int_0^{\mu(R)}\left(g^*\chi_{(0,a)}\right)^{**}(t)^{q'}t^{-q'}V(t)^{-q'}v(t)\,\d t\right)^{\frac1{q'}}\\
&\quad+V(\mu(R))^{-q'}\int_0^ag^*(t)\,\d t
\end{align*}
for every $g\in Z_1'(R, \mu)$ and $a\in(0,\mu(R))$. Hence the claim follows from Lebesgue's dominated convergence theorem.
\end{proof}

\begin{theorem}\label{thm:paransq2lambdaspacesqinfty}
Let $\left(R, \mu\right)$ be a finite, nonatomic measure space. Let $v$ be a weight on $(0,\mu(R))$ satisfying
\begin{equation*}
\tilde{v}(\mu(R))<\infty\quad\text{and}\quad\sup_{t\in(0,\mu(R))}\tilde{v}(t)\frac1{t}\int_0^t\frac1{\tilde{v}(s)}\,\d s<\infty,
\end{equation*}
where $\tilde{v}(t)=\esssup\limits_{s\in(0,t)}v(s)$, $t\in(0,\mu(R)]$. Set $Z(R, \mu)=\Lambda^\infty(v)$. Let $r\in[1,\infty)$ and $w\colon(0,\mu(R))\to(0,\infty)$ be a weight on $(0,\mu(R))$ satisfying \eqref{prel:lambdari} with $v$, $q$ replaced by $w$ and $r$, respectively,
\begin{equation*}
    \lim\limits_{t\to0^+}\frac{W(t)^\frac1{r}}{\tilde{v}(t)}=0
\end{equation*}
and
\begin{equation*}
    \left\|\frac1{\tilde{v}}\right\|_{\Lambda^r(w)}=\infty.
\end{equation*}
Set $Y(R, \mu)=\Lambda^\infty(v) + \Lambda^r(w)$. We have that $\lim\limits_{t\to0^+}\frac{\varphi_{Y}(t)}{\varphi_{Z}(t)}=0$ and $Z(R, \mu)\hookrightarrow Y(R, \mu)$, but the embedding is not almost compact.
\end{theorem}

\begin{proof}
A proof of the theorem can be carried out along the lines of the proof of \hyperref[thm:paransq2lambdaspaces]{Theorem~\ref*{thm:paransq2lambdaspaces}}, and we omit the details. We just note that (cf.~\citep[Lemma~1.5]{GoPi:06})
\begin{equation*}
\|f\|_{Z(R, \mu)}=\esssup_{t\in(0, \mu(R))}f^*(t)\tilde{v}(t)\quad\text{for every $f\in\Mpl(R, \mu)$},
\end{equation*}
and so, since $\frac1{\tilde{v}}$ is a positive, nonincreasing function on $(0,\mu(R))$,
\begin{equation*}
\|g\|_{Z'(R, \mu)}=\sup_{\|f\|_{Z(R, \mu)}\leq1}\int_0^{\mu(R)}f^*(t)g^*(t)\,\d t=\int_0^{\mu(R)}\frac{g^*(t)}{\tilde{v}(t)}\,\d t\quad\text{for every $g\in\Mpl(R, \mu)$}
\end{equation*}
thanks to \eqref{prel:assocnormwithstars}.
\end{proof}

Now is the time to provide concrete examples of spaces $Y(\overline{\Omega}, \nu)$ sought in \hyperref[q:noncompactfundamentallybigger2]{Question~\ref*{q:noncompactfundamentallybigger2}}. We shall do it in the case where $X(\Omega)$ is a Lorentz--Zygmund space.
\begin{theorem}
Let $\Omega\subseteq\rn$ be a bounded Lipschitz domain, $m\in\N$, $m<n$, $d\in[n-m,n]$ and $\nu$ a $d$-Ahlfors measure on $\overline{\Omega}$. Let $q\in(1,\infty]$ and $\alpha\in\R$. Fix any $s\in[1,q)$ and consider the rearrangement-invariant space $Y(\overline{\Omega}, \nu)$ defined as
\begin{equation*}
Y(\overline{\Omega}, \nu) =\begin{cases}
L^{\frac{dp}{n-mp},q;\alpha}(\overline{\Omega}, \nu)+L^{\frac{dp}{n-mp},s;\alpha+\frac1{q}-\frac1{s}}(\overline{\Omega}, \nu)\quad&\text{if $p\in(1,\frac{n}{m})$,}\\
L^{\infty,q;\alpha-1}(\overline{\Omega}, \nu) + L^{\infty,s;\alpha-1+\frac1{q}-\frac1{s},\frac1{q}-\frac1{s}}(\overline{\Omega}, \nu)\quad&\text{if $p=\frac{n}{m}$ and $\alpha<1-\frac1{q}$,}\\
L^{\infty,q;-\frac1{q},-1}(\overline{\Omega}, \nu)+L^{\infty,s;-\frac1{s},\frac1{q}-\frac1{s}-1,\frac1{q}-\frac1{s}}(\overline{\Omega}, \nu)\quad&\text{if $p=\frac{n}{m}$ and $\alpha=1-\frac1{q}$.}
\end{cases}
\end{equation*}

The space $Y(\overline{\Omega}, \nu)$ satisfies
\begin{itemize}
\item $W^mL^{p,q;\alpha}(\Omega)\hookrightarrow Y(\overline{\Omega}, \nu)$ non-compactly,
\item $Y_{L^{p,q;\alpha}}(\overline{\Omega}, \nu)\subsetneq Y(\overline{\Omega}, \nu)$,
\item $\lim\limits_{t\to0^+}\frac{\varphi_Y(t)}{\varphi_{Y_{L^{p,q;\alpha}}}(t)}=0$.
\end{itemize}
%
\end{theorem}
\begin{proof}
Thanks to \citep[Theorem~5.1]{CM:20}, we have that
\begin{equation*}
Y_{L^{p,q;\alpha}}(\overline{\Omega}, \nu)=\begin{cases}
L^{\frac{dp}{n-mp},q;\alpha}(\overline{\Omega}, \nu)\quad&\text{if $p\in(1,\frac{n}{m})$,}\\
L^{\infty,q;\alpha-1}(\overline{\Omega}, \nu)\quad&\text{if $p=\frac{n}{m}$ and $\alpha<1-\frac1{q}$,}\\
L^{\infty,q;-\frac1{q},-1}(\overline{\Omega}, \nu)\quad&\text{if $p=\frac{n}{m}$ and $\alpha=1-\frac1{q}$.}
\end{cases}
\end{equation*}
Due to \citep[Theorem~4.1]{CM:20}, $W^mL^{p,q;\alpha}(\Omega)\hookrightarrow Y(\overline{\Omega}, \nu)$ is not compact if and only if $Y_{L^{p,q;\alpha}}(\overline{\Omega}, \nu)\hookrightarrow Y(\overline{\Omega}, \nu)$ is not almost compact.

The fact that the space $Y(\overline{\Omega}, \nu)$ has all of the desired properties can be derived from \hyperref[thm:paransq2lambdaspaces]{Theorem~\ref*{thm:paransq2lambdaspaces}} (if $q<\infty$) or \hyperref[thm:paransq2lambdaspacesqinfty]{Theorem~\ref*{thm:paransq2lambdaspacesqinfty}} (if $q=\infty$). However, making use of computations that were already done, we can also obtain the assertion directly from \hyperref[prop:paransq2]{Theorem~\ref*{prop:paransq2}} with $X(\overline{\Omega}, \nu)=Y_{L^{p,q;\alpha}}(\overline{\Omega}, \nu)$ and
\begin{equation*}
Z(\overline{\Omega}, \nu)=\begin{cases}
L^{\frac{dp}{n-mp},s;\alpha+\frac1{q}-\frac1{s}}(\overline{\Omega}, \nu)\quad&\text{if $p\in(1,\frac{n}{m})$,}\\
L^{\infty,s;\alpha-1+\frac1{q}-\frac1{s},\frac1{q}-\frac1{s}}(\overline{\Omega}, \nu)\quad&\text{if $p=\frac{n}{m}$ and $\alpha<1-\frac1{q}$,}\\
L^{\infty,s;-\frac1{s},\frac1{q}-\frac1{s}-1,\frac1{q}-\frac1{s}}(\overline{\Omega}, \nu)\quad&\text{if $p=\frac{n}{m}$ and $\alpha=1-\frac1{q}$.}
\end{cases}\end{equation*}
Indeed, we have that $\lim\limits_{t\to0^+}\frac{\varphi_Z(t)}{\varphi_{Y_X}(t)}=0$ thanks to \citep[Lemma~3.7]{OP}, and that $Y_X(\overline{\Omega}, \nu)\not\subseteq Z(\overline{\Omega}, \nu)$ thanks to \citep[Theorem~4.5]{OP}. Finally, since $q>1$, $Y'_{L^{p,q;\alpha}}(\overline{\Omega}, \nu)$ has absolutely continuous norm by \citep[Theorems~6.11 and~9.5]{OP}.
\end{proof}

We have seen that the connection between the optimal target space $Y_X(\overline{\Omega}, \nu)$ in the embedding $W^mX(\Omega)\hookrightarrow Y(\overline{\Omega}, \nu)$ and the question of whether the embedding $W^mX(\Omega)\hookrightarrow Y(\overline{\Omega}, \nu)$ is compact is fairly unrelated to ``the fundamental scale'' of $Y_X(\overline{\Omega}, \nu)$. In fact, we have witnessed that it may even happen that $Y(\overline{\Omega}, \nu)$ is ``fundamentally bigger'' than $M_{Y_X}(\overline{\Omega}, \nu)$ (i.e.~$M_{Y_X}(\overline{\Omega}, \nu)\subsetneq Y(\overline{\Omega}, \nu)$ and $\lim\limits_{t\to0^+}\frac{\varphi_Y(t)}{\varphi_{M_{Y_X}}(t)}=0$), and yet the embedding $W^mX(\Omega)\hookrightarrow Y(\overline{\Omega}, \nu)$ is still not compact. This shows that the question of whether $W^mX(\Omega)\hookrightarrow Y(\overline{\Omega}, \nu)$ is compact is actually far more subtle than it may misleadingly appear when only the Lebesgue (or even two-parameter Lorentz) spaces are taken into account. We conclude this paper with results illustrating even further the unrelatedness of compactness to the fundamental scale of the optimal target space. More precisely, we construct a Lorentz endpoint space $\Lambda_\psi(\overline{\Omega}, \nu)$ (i.e.~the smallest rearrangement-invariant space on the fundamental scale given by $\psi$) and rearrangement-invariant spaces $X(\overline{\Omega}, \nu)$ and $Y_X(\overline{\Omega}, \nu)$ with the following properties:
\begin{itemize}
	\item The spaces $X(\overline{\Omega}, \nu)$ and $Y_X(\overline{\Omega}, \nu)$ are mutually optimal in $W^mX(\Omega)\hookrightarrow Y_X(\overline{\Omega}, \nu)$, that is, $Y_X(\overline{\Omega}, \nu)$ is the optimal target space for $X(\Omega)$ and, simultaneously, $X(\Omega)$ is the optimal domain space for $Y_X(\overline{\Omega}, \nu)$ (i.e.~the largest possible rearrangement-invariant space rendering the embedding true);
	\item $W^mX(\Omega)\hookrightarrow \Lambda_\psi(\overline{\Omega}, \nu)$ non-compactly;
	\item $Y_X(\overline{\Omega}, \nu)\hookrightarrow \Lambda_\psi(\overline{\Omega}, \nu)$;
  \item $\lim\limits_{t\to0^+}\frac{\psi(t)}{\varphi_{Y_X}(t)}=0$.
\end{itemize}

The following proposition characterizes when the spaces in a Sobolev embedding are mutually optimal. We will use the proposition, which is of independent interested, to prove \hyperref[thm:answtoq:noncompactintolorentzendpoint]{Theorem~\ref*{thm:answtoq:noncompactintolorentzendpoint}}, which will tell us how to construct spaces having the properties listed above. We note that the optimal domain space $X_Y(\Omega)$ appearing in the following proposition always exists (see the proof of the proposition).

%

\begin{proposition}\label{prop:optimalpair}
Let $\Omega\subseteq\rn$ be a bounded Lipschitz domain, $m\in\N$, $m<n$, $d\in[n-m,n]$ and $\nu$ a $d$-Ahlfors measure on $\overline{\Omega}$. Let $Y(\overline{\Omega}, \nu)$ be a rearrangement-invariant space over $(\overline{\Omega}, \nu)$. Let $X_Y(\Omega)$ be the optimal domain space for $Y(\overline{\Omega}, \nu)$ in $W^mX(\Omega)\hookrightarrow Y(\overline{\Omega}, \nu)$. The space $Y(\overline{\Omega}, \nu)$ is the optimal range space for $X_Y(\Omega)$ in the Sobolev embedding, that is, $Y(\overline{\Omega}, \nu)$ and $X_Y(\Omega)$ are mutually optimal, if and only if $T_{d,m,n}\colon Y'(0,\nu(\overline{\Omega}))\to Y'(0,\nu(\overline{\Omega}))$ is bounded, where
\begin{equation}\label{prop:optimalpair:defsupop}
T_{d,m,n}f(t)=t^{\frac{n-m}{d}-1}\sup\limits_{s\in[t,\nu(\overline{\Omega}))}s^{1-\frac{n-m}{d}}f^*(s),\ t\in(0,\nu(\overline{\Omega})),\ f\in\M(0,\nu(\overline{\Omega})).
\end{equation}
Moreover, in that case, the norm on $X_Y(\Omega)$ is equivalent to the functional
\begin{equation}\label{prop:optimalpair:optimaldomainfunc}
\M(\Omega)\ni f\mapsto\Bigg\|\int_{\nu(\overline{\Omega})^{1-\frac{n}{d}}t^\frac{n}{d}}^{\nu(\overline{\Omega})}f^*\Big(\frac{|\Omega|}{\nu(\overline{\Omega})}s\Big)s^{-1+\frac{m}{n}}\,\d s\Bigg\|_{Y(0,\nu(\overline{\Omega}))}.
\end{equation}
\end{proposition}
\begin{proof}
First, we note that, for every pair of rearrangement-invariant spaces $Z_1(\Omega)$ and $Z_2(\overline{\Omega}, \nu)$, the validity of the embedding $W^mZ_1(\Omega)\hookrightarrow Z_2(\overline{\Omega}, \nu)$ is equivalent to the validity of
\begin{equation}\label{prop:optimalpair:reductionprinciple}
\Bigg\|\int_{\nu(\overline{\Omega})^{1-\frac{n}{d}}t^\frac{n}{d}}^{\nu(\overline{\Omega})}f^*\Big(\frac{|\Omega|}{\nu(\overline{\Omega})}s\Big)s^{-1+\frac{m}{n}}\,\d s\Bigg\|_{Z_2(0,\nu(\overline{\Omega}))}\lesssim\|f^*\|_{Z_1(0,|\Omega|)}\quad\text{for every $f\in\M(\Omega)$}
\end{equation}
thanks to \citep[Theorems~4.1~and~4.3]{CPS:20}. Using this characterization, we readily obtain that the rearrangement-invariant space $X_Y(\Omega)$ induced by the rearrangement-invariant function norm
\begin{equation}\label{prop:optimalpair:optimaldomainnorm}
\Mpl(\Omega)\ni f\mapsto\sup\limits_{g\sim f}\Bigg\|\int_{\nu(\overline{\Omega})^{1-\frac{n}{d}}t^\frac{n}{d}}^{\nu(\overline{\Omega})}g\Big(\frac{|\Omega|}{\nu(\overline{\Omega})}s\Big)s^{-1+\frac{m}{n}}\,\d s\Bigg\|_{Y(0,\nu(\overline{\Omega}))},
\end{equation}
where the supremum is taken over all $g\in\Mpl(0,|\Omega|)$ equimeasurable with $f$, is the optimal domain space for $Y(\overline{\Omega}, \nu)$ in $W^mX(\Omega)\hookrightarrow Y(\overline{\Omega}, \nu)$. The fact that \eqref{prop:optimalpair:optimaldomainnorm} is indeed a rearrangement-invariant function norm can be proved similarly to \citep[Theorem~3.3]{KePi:06}.

Second, assume that the space $Y(\overline{\Omega}, \nu)$ is the optimal target space for $X_Y(\Omega)$ in the Sobolev embedding $W^mX_Y(\Omega)\hookrightarrow Z(\overline{\Omega}, \nu)$. Set $\alpha=1-\frac{n-m}{d}\in[0,1)$ and $T=T_{d,n,m}$. We have that
\begin{equation}\label{prop:optimalpair:optrangeforoptdomain}
\|g\|_{Y'(0,\nu(\overline{\Omega}))}\approx\Bigg\|s^{-1+\frac{m}{n}}\int_0^{|\Omega|^{1-\frac{d}{n}}s^\frac{d}{n}}g^*\Big(\frac{\nu(\overline{\Omega})}{|\Omega|}t\Big)\,\d t\Bigg\|_{X_Y'(0,|\Omega|)}\quad\text{for every $g\in\M(0,\nu(\overline{\Omega}))$}
\end{equation}
thanks to \citep[Theorem~4.4]{CPS:20}. Observe that, for every $g\in\M(0,\nu(\overline{\Omega}))$, the function $Tg$ is nonincreasing on $(0, \nu(\overline{\Omega}))$. Consequently, for every $g\in\M(0,\nu(\overline{\Omega}))$,
\begin{equation}\label{prop:optimalpair:odhadnasupremalni}
\begin{aligned}
\|Tg\|_{Y'(0,\nu(\overline{\Omega}))}&\approx\Bigg\|s^{-1+\frac{m}{n}}\int_0^{|\Omega|^{1-\frac{d}{n}}s^\frac{d}{n}}Tg\Big(\frac{\nu(\overline{\Omega})}{|\Omega|}t\Big)\,\d t\Bigg\|_{X_Y'(0,|\Omega|)}\\
&\lesssim\Bigg\|s^{-1+\frac{m}{n}}\int_0^{|\Omega|^{1-\frac{d}{n}}s^\frac{d}{n}}t^{-\alpha}\sup\limits_{\tau\in\big[t,|\Omega|^{1-\frac{d}{n}}s^\frac{d}{n}\big]}\tau^\alpha g^*\Big(\frac{\nu(\overline{\Omega})}{|\Omega|}\tau\Big)\,\d t\Bigg\|_{X_Y'(0,|\Omega|)}\\
&\quad+\Bigg\|s^{-1+\frac{m}{n}}\int_0^{|\Omega|^{1-\frac{d}{n}}s^\frac{d}{n}}t^{-\alpha}\sup\limits_{\tau\in\big[|\Omega|^{1-\frac{d}{n}}s^\frac{d}{n},|\Omega|\big)}\tau^\alpha g^*\Big(\frac{\nu(\overline{\Omega})}{|\Omega|}\tau\Big)\,\d t\Bigg\|_{X_Y'(0,|\Omega|)}.
\end{aligned}
\end{equation}
As for the first term, using \citep[Lemma~3.3(ii)]{CP:16}, we obtain that
\begin{align*}
&\Bigg\|s^{-1+\frac{m}{n}}\int_0^{|\Omega|^{1-\frac{d}{n}}s^\frac{d}{n}}t^{-\alpha}\sup\limits_{\tau\in\big[t,|\Omega|^{1-\frac{d}{n}}s^\frac{d}{n}\big]}\tau^\alpha g^*\Big(\frac{\nu(\overline{\Omega})}{|\Omega|}\tau\Big)\,\d t\Bigg\|_{X_Y'(0,|\Omega|)}\\
&\lesssim\Bigg\|s^{-1+\frac{m}{n}}\int_0^{|\Omega|^{1-\frac{d}{n}}s^\frac{d}{n}}g^*\Big(\frac{\nu(\overline{\Omega})}{|\Omega|}t\Big)\,\d t\Bigg\|_{X_Y'(0,|\Omega|)}
\end{align*}
for every $g\in\M(0,\nu(\overline{\Omega}))$. As for the second term in \eqref{prop:optimalpair:odhadnasupremalni}, we have that
\begin{align*}
&\Bigg\|s^{-1+\frac{m}{n}}\int_0^{|\Omega|^{1-\frac{d}{n}}s^\frac{d}{n}}t^{-\alpha}\sup\limits_{\tau\in\big[|\Omega|^{1-\frac{d}{n}}s^\frac{d}{n},|\Omega|\big)}\tau^\alpha g^*\Big(\frac{\nu(\overline{\Omega})}{|\Omega|}\tau\Big)\,\d t\Bigg\|_{X_Y'(0,|\Omega|)}\\
&\approx\Bigg\|s^{-1+\frac{m}{n}+\frac{d}{n}(1-\alpha)}\sup\limits_{\tau\in\big[|\Omega|^{1-\frac{d}{n}}s^\frac{d}{n},|\Omega|\big)}\tau^\alpha g^*\Big(\frac{\nu(\overline{\Omega})}{|\Omega|}\tau\Big)\Bigg\|_{X_Y'(0,|\Omega|)}\\
&\approx\Bigg\|\sup\limits_{\tau\in\big[s,|\Omega|\big)}\tau^{\alpha\frac{d}{n}} g^*\Bigg(\frac{\nu(\overline{\Omega})}{|\Omega|^\frac{d}{n}}\tau^\frac{d}{n}\Bigg)\Bigg\|_{X_Y'(0,|\Omega|)}\lesssim\Bigg\|s^{\alpha\frac{d}{n}} g^*\Bigg(\frac{\nu(\overline{\Omega})}{|\Omega|^\frac{d}{n}}s^\frac{d}{n}\Bigg)\Bigg\|_{X_Y'(0,|\Omega|)}\\
&\leq\Bigg\|s^{\alpha\frac{d}{n}} g^{**}\Bigg(\frac{\nu(\overline{\Omega})}{|\Omega|^\frac{d}{n}}s^\frac{d}{n}\Bigg)\Bigg\|_{X_Y'(0,|\Omega|)}\approx\Bigg\|s^{-1+\frac{m}{n}}\int_0^{|\Omega|^{1-\frac{d}{n}}s^\frac{d}{n}}g^*\Big(\frac{\nu(\overline{\Omega})}{|\Omega|}t\Big)\,\d t\Bigg\|_{X_Y'(0,|\Omega|)}
\end{align*}
for every $g\in\M(0,\nu(\overline{\Omega}))$, where the last but one inequality follows from a simple modification of \citep[Lemma~4.10]{EMMP:20}. Hence, combining these two chains of inequalities with \eqref{prop:optimalpair:odhadnasupremalni} and \eqref{prop:optimalpair:optrangeforoptdomain}, we obtain that $T$ is bounded on $Y'(0,\nu(\overline{\Omega}))$.

Finally, assume that $T$ is bounded on $Y'(0,\nu(\overline{\Omega}))$. Let $Z(\overline{\Omega}, \nu)$ be the optimal target space for $X_Y(\Omega)$ in the Sobolev embedding $W^mX_Y(\Omega)\hookrightarrow Z(\overline{\Omega}, \nu)$. Its existence is guaranteeed by \citep[Theorem~4.4]{CPS:20}. Since $X_Y(\Omega)$ is the optimal domain space for $Y(\overline{\Omega}, \nu)$ and $Z(\overline{\Omega}, \nu)$ is the optimal target space for $X_Y(\Omega)$, we have that $Z(\overline{\Omega}, \nu)\hookrightarrow Y(\overline{\Omega}, \nu)$. Hence, for every $f\in\M(\Omega)$,
\begin{equation}\label{prop:optimalpair:equiYZLHS}
\begin{aligned}
&\Bigg\|\int_{\nu(\overline{\Omega})^{1-\frac{n}{d}}t^\frac{n}{d}}^{\nu(\overline{\Omega})}f^*\Big(\frac{|\Omega|}{\nu(\overline{\Omega})}s\Big)s^{-1+\frac{m}{n}}\,\d s\Bigg\|_{Y(0,\nu(\overline{\Omega}))}\\
&\lesssim\Bigg\|\int_{\nu(\overline{\Omega})^{1-\frac{n}{d}}t^\frac{n}{d}}^{\nu(\overline{\Omega})}f^*\Big(\frac{|\Omega|}{\nu(\overline{\Omega})}s\Big)s^{-1+\frac{m}{n}}\,\d s\Bigg\|_{Z(0,\nu(\overline{\Omega}))}\\
&\lesssim\|f^*\|_{X_Y(0,|\Omega|)} = \|f\|_{X_Y(\Omega)}
\end{aligned}
\end{equation}
thanks to \eqref{prop:optimalpair:reductionprinciple} with $Z_1(\Omega)=X_Y(\Omega)$ and $Z_2(\overline{\Omega}, \nu)=Z(\overline{\Omega}, \nu)$. Now, fix any $f\in\M(\Omega)$ and let $g\in\Mpl(0,|\Omega|)$ be equimeasurable with $f$. Observe that, for each $h\in\M(0,\nu(\overline{\Omega}))$, the function
\begin{equation*}
(0,|\Omega|)\ni s\mapsto s^{-1+\frac{m}{n}}\int_0^{|\Omega|^{1-\frac{d}{n}}s^\frac{d}{n}}Th\Big(\frac{\nu(\overline{\Omega})}{|\Omega|}t\Big)\,\d t
\end{equation*}
is nonincreasing on $(0,|\Omega|)$, for it is a constant multiple of the integral mean of a nonincreasing function with respect to the measure $t^{-\alpha}\,\d t$. Using the boundedness of $T$ on $Y'(0,\nu(\overline{\Omega}))$ in the final inequality, we obtain that
\begin{align*}
&\Bigg\|\int_{\nu(\overline{\Omega})^{1-\frac{n}{d}}t^\frac{n}{d}}^{\nu(\overline{\Omega})}g\Big(\frac{|\Omega|}{\nu(\overline{\Omega})}s\Big)s^{-1+\frac{m}{n}}\,\d s\Bigg\|_{Y(0,\nu(\overline{\Omega}))}\\
&=\sup\limits_{\|h\|_{Y'(0,\nu(\overline{\Omega}))}\leq1}\int_0^{\nu(\overline{\Omega})}h^*(t)\int_{\nu(\overline{\Omega})^{1-\frac{n}{d}}t^\frac{n}{d}}^{\nu(\overline{\Omega})}g\Big(\frac{|\Omega|}{\nu(\overline{\Omega})}s\Big)s^{-1+\frac{m}{n}}\,\d s\,\d t\\
&\approx\sup\limits_{\|h\|_{Y'(0,\nu(\overline{\Omega}))}\leq1}\int_0^{|\Omega|}g(s)s^{-1+\frac{m}{n}}\int_0^{|\Omega|^{1-\frac{d}{n}}s^\frac{d}{n}}h^*\Big(\frac{\nu(\overline{\Omega})}{|\Omega|}t\Big)\,\d t\,\d s\\
&\lesssim\sup\limits_{\|h\|_{Y'(0,\nu(\overline{\Omega}))}\leq1}\int_0^{|\Omega|}g(s)s^{-1+\frac{m}{n}}\int_0^{|\Omega|^{1-\frac{d}{n}}s^\frac{d}{n}}Th\Big(\frac{\nu(\overline{\Omega})}{|\Omega|}t\Big)\,\d t\,\d s\\
&\leq\sup\limits_{\|h\|_{Y'(0,\nu(\overline{\Omega}))}\leq1}\int_0^{|\Omega|}f^*(s)s^{-1+\frac{m}{n}}\int_0^{|\Omega|^{1-\frac{d}{n}}s^\frac{d}{n}}Th\Big(\frac{\nu(\overline{\Omega})}{|\Omega|}t\Big)\,\d t\,\d s\\
&\approx\sup\limits_{\|h\|_{Y'(0,\nu(\overline{\Omega}))}\leq1}\int_0^{\nu(\overline{\Omega})}Th(t)\int_{\nu(\overline{\Omega})^{1-\frac{n}{d}}t^\frac{n}{d}}^{\nu(\overline{\Omega})}f^*\Big(\frac{|\Omega|}{\nu(\overline{\Omega})}s\Big)s^{-1+\frac{m}{n}}\,\d s\,\d t\\
&\leq\Bigg\|\int_{\nu(\overline{\Omega})^{1-\frac{n}{d}}t^\frac{n}{d}}^{\nu(\overline{\Omega})}f^*\Big(\frac{|\Omega|}{\nu(\overline{\Omega})}s\Big)s^{-1+\frac{m}{n}}\,\d s\Bigg\|_{Y(0,\nu(\overline{\Omega}))}\sup\limits_{\|h\|_{Y'(0,\nu(\overline{\Omega}))}\leq1}\|Th\|_{Y'(0,\nu(\overline{\Omega}))}\\
&\lesssim\Bigg\|\int_{\nu(\overline{\Omega})^{1-\frac{n}{d}}t^\frac{n}{d}}^{\nu(\overline{\Omega})}f^*\Big(\frac{|\Omega|}{\nu(\overline{\Omega})}s\Big)s^{-1+\frac{m}{n}}\,\d s\Bigg\|_{Y(0,\nu(\overline{\Omega}))},
\end{align*}
where the equality follows from \eqref{prel:assocnormwithstars}, the equivalences follow from Fubini's theorem coupled with a change of variables, the second inequality follows from \eqref{prel:HL} and the equimeasurability of $g$ with $f$, and the last but one inequality follows from \eqref{prel:holder}. Hence
\begin{equation}\label{prop:optimalpair:equiYZRHS}
\|f\|_{X_Y(\Omega)}\lesssim\Bigg\|\int_{\nu(\overline{\Omega})^{1-\frac{n}{d}}t^\frac{n}{d}}^{\nu(\overline{\Omega})}f^*\Big(\frac{|\Omega|}{\nu(\overline{\Omega})}s\Big)s^{-1+\frac{m}{n}}\,\d s\Bigg\|_{Y(0,\nu(\overline{\Omega}))}
\end{equation}
for every $f\in\M(\Omega)$. Combining \eqref{prop:optimalpair:equiYZLHS} and \eqref{prop:optimalpair:equiYZRHS}, we obtain that
\begin{equation}\label{prop:optimalpair:equiYZ}
\Bigg\|\int_{\nu(\overline{\Omega})^{1-\frac{n}{d}}t^\frac{n}{d}}^{\nu(\overline{\Omega})}f^*\Big(\frac{|\Omega|}{\nu(\overline{\Omega})}s\Big)s^{-1+\frac{m}{n}}\,\d s\Bigg\|_{Y(0,\nu(\overline{\Omega}))}\approx\Bigg\|\int_{\nu(\overline{\Omega})^{1-\frac{n}{d}}t^\frac{n}{d}}^{\nu(\overline{\Omega})}f^*\Big(\frac{|\Omega|}{\nu(\overline{\Omega})}s\Big)s^{-1+\frac{m}{n}}\,\d s\Bigg\|_{Z(0,\nu(\overline{\Omega}))}
\end{equation}
for every $f\in\M(\Omega)$. Let $v$ be a simple function on $(0,\nu(\overline{\Omega}))$ having the form $v=\sum\limits_{j=0}^Nc_j\chi_{(0,t_j)}$, where $N\in\N$, $c_j>0$, $0<t_1<\cdots<t_N<\nu(\overline{\Omega})$. By straightforwardly modifying \citep[Lemma~4.9]{EMMP:20} and using \eqref{prop:optimalpair:equiYZ}, we have that
\begin{equation}\label{prop:optimalpair:YZonsimple}
\|v\|_{Y(0,\nu(\overline{\Omega}))}\approx\|v\|_{Z(0,\nu(\overline{\Omega}))}.
\end{equation}
Since every nonnegative, nonincreasing function on $(0,\nu(\overline{\Omega}))$ is the limit of a nondecreasing sequence of such functions, it follows from \eqref{prop:optimalpair:YZonsimple} and property (P3) of rearrangement-invariant function norms that
\begin{equation*}
\|g^*\|_{Y(0, \nu(\overline{\Omega}))}\approx\|g^*\|_{Z(0, \nu(\overline{\Omega}))}
\end{equation*}
for every $g\in\M(\overline{\Omega}, \nu)$. Hence $Y(\overline{\Omega}, \nu)=Z(\overline{\Omega}, \nu)$, that is, $Y(\overline{\Omega}, \nu)$ is the optimal target space for $X_Y(\Omega)$.
\end{proof}

\begin{remark}\label{rem:boundednessofT}
Note that, in the special case where $d=n-m$, the operator $T_{d,m,n}$ collapses into the identity operator, and thus it is plainly bounded on any rearrangement-invariant space.

When $Y(\overline{\Omega}, \nu)=L^{p,q;\alpha}(\overline{\Omega}, \nu)$ is a Lorentz--Zygmund space, the operator $T_{d,m,n}$ is bounded on $Y'(0, \nu(\overline{\Omega}))$ if and only if $p\in(\frac{d}{n-m},\infty]$, or $p=\frac{d}{n-m}$, $q=1$ and $\alpha\geq0$ (see~\cite[Proposition~5.4]{Mih:20}, cf.~\citep[Theorem~3.2]{GOP}).
\end{remark}

\begin{theorem}\label{thm:answtoq:noncompactintolorentzendpoint}
Let $\Omega\subseteq\rn$ be a bounded Lipschitz domain, $m\in\N$, $m<n$, $d\in[n-m,n]$ and $\nu$ a $d$-Ahlfors measure on $\overline{\Omega}$. Assume that $Z(\overline{\Omega}, \nu)$ is a rearrangement-invariant space over
$(\overline{\Omega}, \nu)$ and $\Lambda_{\psi}(\overline{\Omega}, \nu)$ is a Lorentz endpoint space such that $\lim\limits_{t\to0^+}\frac{\psi(t)}{\varphi_{Z}(t)} = 0$. Assume that
$Z(\overline{\Omega}, \nu)\not\subseteq\Lambda_{\psi}(\overline{\Omega}, \nu)$. Furthermore, assume that the operator $T_{d,m,n}$, defined by \eqref{prop:optimalpair:defsupop}, is bounded on both
$M_{\frac{t}{\psi(t)}}(0,\nu(\overline{\Omega}))$ and $Z'(0,\nu(\overline{\Omega}))$. Let $X_{\Lambda_{\psi}}(\Omega)$ and $X_{Z}(\Omega)$ be the optimal domain spaces in $W^{m}X(\Omega)\hookrightarrow
\Lambda_{\psi}(\overline{\Omega}, \nu)$ and $W^{m}X(\Omega)\hookrightarrow Z(\overline{\Omega}, \nu)$, respectively. Set $X(\Omega)=X_{\Lambda_{\psi}}(\Omega)\cap X_{Z}(\Omega)$ and
$Y_X(\overline{\Omega}, \nu)=\Lambda_{\psi}(\overline{\Omega}, \nu)\cap Z(\overline{\Omega}, \nu)$. The following facts are true:
\begin{itemize}
\item The spaces $X(\Omega)$ and $Y_X(\overline{\Omega}, \nu)$ are mutually optimal in $W^mX(\Omega)\hookrightarrow Y_X(\overline{\Omega}, \nu)$;
\item $W^mX(\Omega)\hookrightarrow \Lambda_\psi(\overline{\Omega}, \nu)$ non-compactly;
\item $Y_X(\overline{\Omega}, \nu)\hookrightarrow\Lambda_{\psi}(\overline{\Omega}, \nu)$;
\item $\lim\limits_{t\to0^+}\frac{\psi(t)}{\varphi_{Y_X}(t)}=0$.
\end{itemize}
\end{theorem}
\begin{proof}
First, as $\Lambda_{\psi}(\overline{\Omega}, \nu)$ has absolutely continuous norm, we have that $\left(\Lambda_{\psi}(\overline{\Omega}, \nu)\cap Z(\overline{\Omega}, \nu)\right)'=M_{\frac{t}{\psi(t)}}(\overline{\Omega}, \nu)+Z'(\overline{\Omega}, \nu)$ (see~\citep[Chapter~3, Exercise~5]{BS}). Since the supremum operator $T_{d,m,n}$ is bounded on both $M_{\frac{t}{\psi(t)}}(0,\nu(\overline{\Omega}))$ and $Z'(0,\nu(\overline{\Omega}))$, it is also bounded on $M_{\frac{t}{\psi(t)}}(0,\nu(\overline{\Omega})) + Z'(0,\nu(\overline{\Omega}))$. Thanks to \hyperref[prop:optimalpair]{Proposition~\ref*{prop:optimalpair}}, we have that the spaces in the embeddings $W^mX_{\Lambda_{\psi}\cap Z}(\Omega)\hookrightarrow\Lambda_{\psi}(\overline{\Omega}, \nu)\cap Z(\overline{\Omega}, \nu)$, $W^mX_{\Lambda_{\psi}}(\Omega)\hookrightarrow\Lambda_{\psi}(\overline{\Omega}, \nu)$ and $W^mX_{Z}(\Omega)\hookrightarrow Z(\overline{\Omega}, \nu)$ are mutually optimal in each of these embeddings. Moreover, it follows from \eqref{prop:optimalpair:optimaldomainfunc} that $X_{\Lambda_{\psi}\cap Z}(\Omega)=X_{\Lambda_{\psi}}(\Omega)\cap X_{Z}(\Omega)$. Hence the spaces $X(\Omega)$ and $Y_X(\overline{\Omega}, \nu)$ are mutually optimal in the embedding $W^mX(\Omega)\hookrightarrow Y_X(\overline{\Omega}, \nu)$, where $X(\Omega)=X_{\Lambda_{\psi}}(\Omega)\cap X_{Z}(\Omega)$ and $Y_X(\overline{\Omega}, \nu)=\Lambda_{\psi}(\overline{\Omega}, \nu)\cap Z(\overline{\Omega}, \nu)$.

Next, as $\lim\limits_{t\to0^+}\frac{\psi(t)}{\varphi_{Z}(t)} = 0$ and $\varphi_{Y_X}\approx\max\{\psi,\varphi_Z\}$, we have that $\lim\limits_{t\to0^+}\frac{\psi(t)}{\varphi_{Y_X}(t)}=0$.

Last, thanks to the fact that $Z(\overline{\Omega}, \nu)\not\subseteq\Lambda_{\psi}(\overline{\Omega}, \nu)$, the embedding $\Lambda_{\psi}(\overline{\Omega}, \nu)\cap Z(\overline{\Omega}, \nu)\hookrightarrow\Lambda_{\psi}(\overline{\Omega}, \nu)$ is not almost compact (\citep[Lemma~3.7]{F-MMP:10}), and so $W^{m}X(\Omega)\hookrightarrow \Lambda_\psi(\overline{\Omega}, \nu)$ non-compactly owing to \citep[Theorem~4.1]{CM:20}.
\end{proof}

We conclude this paper with a concrete application of \hyperref[thm:answtoq:noncompactintolorentzendpoint]{Theorem~\ref*{thm:answtoq:noncompactintolorentzendpoint}}, which covers a relatively wide class of function spaces.

\begin{theorem}
Let $\Omega\subseteq\rn$ be a bounded Lipschitz domain, $m\in\N$, $m<n$, $d\in[n-m,n]$ and $\nu$ a $d$-Ahlfors measure on $\overline{\Omega}$. Let $p\in(\frac{d}{n-m},\infty]$ and $q\in(1,\infty]$. Let $\alpha\in\R$ if $p<\infty$, or $\alpha+1<0$ if $p=\infty$. Set
\begin{equation*}
X(\Omega)=\begin{cases}
L^{\frac{np}{d+mp},1;\alpha}(\Omega)\cap L^{\frac{np}{d+mp},q;\alpha+1-\frac1{q}}(\Omega)\quad&\text{if $p\in(\frac{d}{n-m},\infty)$,}\\
L^{\frac{n}{m},1;\alpha+1}(\Omega)\cap X_q(\Omega)\quad&\text{if $p=\infty$ and $\alpha+1<0$},
\end{cases}
\end{equation*}
where $X_q(\Omega)$ is the rearrangement-invariant space over $\Omega$ satisfying
\begin{equation*}
\|f\|_{X_q(\Omega)}\approx\Big\|t^{-\frac1{q}}\ell^{\alpha+1-\frac1{q}}(t)\ell\ell^{1-\frac1{q}}(t)\int_{|\Omega|^{1-\frac{n}{d}}t^\frac{n}{d}}^{|\Omega|}f^*(s)s^{-1+\frac{m}{n}}\,\d s\Big\|_{L^q(0,|\Omega|)}.
\end{equation*}
Set
\begin{equation*}
Y_X(\overline{\Omega}, \nu)=\begin{cases}
L^{p,1;\alpha}(\overline{\Omega}, \nu)\cap L^{p,q;\alpha+1-\frac1{q}}(\overline{\Omega}, \nu)\quad&\text{if $p\in(\frac{d}{n-m},\infty)$,}\\
L^{\infty,1;\alpha}(\overline{\Omega}, \nu)\cap L^{\infty, q;\alpha+1-\frac1{q},1-\frac1{q}}(\overline{\Omega}, \nu)\quad&\text{if $p=\infty$ and $\alpha+1<0$}.
\end{cases}
\end{equation*}
The following facts are true:
\begin{itemize}
\item The spaces $X(\Omega)$ and $Y_X(\overline{\Omega}, \nu)$ are mutually optimal in $W^mX(\Omega)\hookrightarrow Y_X(\overline{\Omega}, \nu)$;
\item $W^mX(\Omega)\hookrightarrow \Lambda_\psi(\overline{\Omega}, \nu)$ non-compactly, where $\Lambda_\psi(\overline{\Omega}, \nu)$ is the Lorentz endpoint space whose fundamental function is equivalent to
\begin{equation}\label{thm:answtoq:noncompactintolorentzendpointLZ:psi}
\psi(t)\approx\begin{cases}
t^\frac1{p}\ell(t)^{\alpha}\quad&\text{if $p\in(\frac{d}{n-m},\infty)$,}\\
\ell\ell(t)^{\alpha+1}\quad&\text{if $p=\infty$ and $\alpha+1<0$;}
\end{cases}
\end{equation}
\item $Y_X(\overline{\Omega}, \nu)\hookrightarrow\Lambda_\psi(\overline{\Omega}, \nu)$;
\item $\lim\limits_{t\to0^+}\frac{\psi(t)}{\varphi_{Y_X}(t)}=0$.
\end{itemize}
\end{theorem}
\begin{proof}
The assertion follows from \hyperref[thm:answtoq:noncompactintolorentzendpoint]{Theorem~\ref*{thm:answtoq:noncompactintolorentzendpoint}} with
\begin{equation*}
Z(\overline{\Omega}, \nu)=\begin{cases}
L^{p,q;\alpha+1-\frac1{q}}(\overline{\Omega}, \nu)\quad&\text{if $p\in(\frac{d}{n-m},\infty)$,}\\
L^{\infty, q;\alpha+1-\frac1{q},1-\frac1{q}}(\overline{\Omega}, \nu)\quad&\text{if $p=\infty$ and $\alpha+1<0$,}
\end{cases}
\end{equation*}
and $\psi$ equal to a concave function on $[0,\nu(\overline{\Omega})]$ satisfying \eqref{thm:answtoq:noncompactintolorentzendpointLZ:psi}. Note that $\Lambda_\psi(\overline{\Omega}, \nu)=L^{p,1;\alpha}(\overline{\Omega}, \nu)$.

As for the assumptions of \hyperref[thm:answtoq:noncompactintolorentzendpoint]{Theorem~\ref*{thm:answtoq:noncompactintolorentzendpoint}} being satisfied, since (e.g.~\cite[Lemma~3.14]{OP})
\begin{equation*}
\varphi_{Z}(t)\approx\begin{cases}
t^\frac1{p}\ell(t)^{\alpha+1-\frac1{q}}\quad&\text{if $p\in(\frac{d}{n-m},\infty)$,}\\
\ell(t)^{\alpha+1}\ell\ell(t)^{1-\frac1{q}}\quad&\text{if $p=\infty$ and $\alpha+1<0$,}
\end{cases}
\end{equation*}
it follows that $\lim\limits_{t\to0^+}\frac{\psi(t)}{\varphi_{Z}(t)} = 0$. Furthermore, we have that $Z(\overline{\Omega}, \nu)\not\subseteq\Lambda_{\psi}(\overline{\Omega}, \nu)$ thanks to \cite[Theorem~4.5]{OP}. Next, as $p>\frac{d}{n-m}$, the supremum operator $T_{d,m,n}$ is bounded on the associate spaces of $Z(0, \nu(\overline{\Omega}))$ and $L^{p,1;\alpha}(0, \nu(\overline{\Omega}))$ (see \hyperref[rem:boundednessofT]{Remark~\ref*{rem:boundednessofT}}). Finally, the description of $X_{L^{p,1;\alpha}}(\Omega)$ and $X_Z(\Omega)$ can be obtained similarly to \citep[Thereom~5.3]{Mih:20}.
\end{proof}

\paragraph{Acknowledgments}
The second author would like to express his sincere gratitude to the Fulbright Program for supporting him and giving him the opportunity to visit the first author at the Ohio State University as a Fulbrighter. The research we conducted during the visit led us to the questions this paper deals with.

\bibliography{non-compact-embeddings}
\end{document}